\documentclass[11pt,reqno]{amsart}

\usepackage{amsmath, amsthm, amsfonts,bbm}
\usepackage[utf8]{inputenc}
\usepackage{srcltx}
\usepackage{enumerate}
\usepackage[hypertex,colorlinks=true,linkcolor=blue,citecolor=blue,]{hyperref}

\usepackage{mathrsfs}
\usepackage{amsmath,amsthm,amssymb}
\usepackage{latexsym}
\usepackage{times}
\usepackage{enumerate}
\usepackage{mathrsfs}
\usepackage{stmaryrd}
\usepackage{amsopn}
\usepackage{amsmath}
\usepackage{amssymb}
\usepackage{amsfonts}
\usepackage{amsbsy}
\usepackage{amscd,indentfirst}
\usepackage{hyperref}
\usepackage{amsfonts,amsmath,latexsym,amssymb,verbatim,amsbsy}
\usepackage{amsthm}
\usepackage{colordvi}
\usepackage{dsfont}
\usepackage{pstricks}
\usepackage{mathrsfs}
\usepackage{amsmath,amsthm,amssymb}
\usepackage{latexsym}
\usepackage{times}
\usepackage{enumerate}
\usepackage{mathrsfs}
\usepackage{stmaryrd}
\usepackage{amsopn}
\usepackage{amsmath}
\usepackage{amssymb}
\usepackage{amsfonts}
\usepackage{amsbsy}
\usepackage{amscd,indentfirst}
\usepackage{hyperref}
\usepackage{amsfonts,amsmath,latexsym,amssymb,verbatim,amsbsy}
\usepackage{amsthm}
\usepackage{colordvi}

\newtheorem{theo}{\sc Theorem}[section]
\newtheorem{lemme}[theo]{\sc Lemma}
\newtheorem{propo}[theo]{\sc Proposition}
\newtheorem{cor}[theo]{\sc Corollary}
\newtheorem{hyp}[theo]{\sc Assumption}
\newtheorem{nb}[theo]{\sc Remark}
\newtheorem{defi}[theo]{\sc Definition}

\theoremstyle{definition}

\def \leq {\leqslant}
\def \geq {\geqslant}
\numberwithin{equation}{section}

\def \x {\mathbf{x}}

\def \G {\mathcal{G}}

\def \d {\mathrm{d }}
\def \D {\mathscr{D}}

\def \R {\mathscr{R}}

\def \com {$C_0$-semigroup }

\def \Tt {(\T(t))_{t \geq 0}}

\def \Q {\mathcal{Q}}

\def \T {\mathcal{T}}
\def \B {\mathcal{B}}

\def \t {\tau}
\def \u {\mathcal{U}}
\def \U {\mathcal{U}}
\def \V {\mathcal{V}}
\def \O {\mathbf{\Omega}}
\def \X {\mathfrak{X}}
\def \la {\left \langle}
\def \ra {\right \rangle}

\def \l {\lambda}
\def \A {\mathcal{A}}
\def \cc {\mathfrak{a}}

\def \T {\mathcal{T}}

\def \e {\mathcal{E}}
\def \U {\mathbf{U}}
\def \V {\mathbf{V}}

\def \uts {\{\U(t,s)\}_{t\geq s}}

\def \Z {\mathcal{Z}}
\def \fB {\mathfrak{B}}

\def \vt {(\mathbf{V}(t,s))_{t \geq s}}

\def \R {\mathbb{R}}

\title[]{Non-autonomous Honesty
theory in abstract state spaces with applications to linear kinetic equations}

\author{L. Arlotti, B. Lods \& M. Mokhtar-Kharroubi}

\address{\textbf{{L. Arlotti}},  {Dipartimento di Ingegneria Civile}, Universit\`a di
Udine, via delle Scienze 208,  33100 Udine, Italy. \newline{\tt
luisa.arlotti@uniud.it}}

\address{\textbf{{B. Lods}}, Universit\`{a} degli
  Studi di Torino \& Collegio Carlo Alberto, Department of Statistics and Economics, Corso Unione Sovietica, 218/bis, 10134 Torino,
  Italy.\newline
  \noindent{\tt lods@econ.unito.it}}
\address{\textbf{{M. Mokhtar-Kharroubi}}, Universit\'e de Franche--Comt\'e, Equipe de Math\'ematiques, CNRS UMR
6623, 16, route de Gray, 25030 Besan\c con Cedex, France.\newline
\noindent{\tt mmokhtar@univ-fcomte.fr}}

\begin{document}

\medskip

\begin{abstract} We provide a honesty theory  of substochastic evolution families in real abstract state space, extending to an non-autonomous setting the result obtained for $C_0$-semigroups in our recent contribution \textit{[On perturbed substochastic semigroups in abstract state spaces, \textit{Z. Anal. Anwend.} \textbf{30}, 457--495, 2011]}. The link with the honesty theory of perturbed substochastic semigroups is  established. Several applications to non-autonomous linear kinetic equations (linear Boltzmann equation and fragmentation equation) are provided.
\end{abstract}
\date{}

\maketitle

%
%
%
%
\section{Introduction}

Perturbation theory in $L^1$ spaces goes back to a classical and seminal paper
by T. Kato \cite{kato} on Kolmogorov's differential equations and is strongly
tied to positivity and to convergence monotone arguments. This theory had
a renewal of interest in the 70's when E. B. Davies \cite{Da2} adapted it to
the (non-commutative) context of quantum dynamical semigroups. Later, a
new impetus came from various problems from kinetic theory \cite{arlotti1, voigt1} and gave rise to a general formalism, called "honesty theory" of perturbed substochatic semigroups \cite{banlach, thieme, BA1, fvdm}. Such a honesty theory
provides a "resolvent perturbation approach" characterizing the generator of the
underlying perturbed semigroup by means of a suitable functional; this theory
is described in the book \cite{arloban}. Afterwards, this honesty theory
was improved and extended in different directions in \cite{mkvo} while
a noncommutative version of \cite{mkvo} in the Banach space of trace
class operators was given in \cite{class}. In a recent paper  \cite{ALM}, we provided a general honesty theory of perturbed substochastic semigroup in  abstract state spaces (i.e. ordered Banach spaces with additive norm on
the positive cone), unifying in one hand the results of \cite{mkvo} and \cite{class}
and in another hand providing another (equivalent) approach to the honesty which relies on Dyson-Phillips expansions rather
than the usual resolvent approach; we refer to \cite{ALM} for more
precise information and also for much more references. Our aim here is to built
a \emph{non-autonomous version} of this honesty theory and to show how kinetic equations fit into
this theory. Dealing with such non-autonomous (two-parameters) families of operators leads to important technical difficulties and, in particular, prevents us to use (in a direct way) any approach based on resolvent estimates.  We will resort to several of the results developed in our previous contribution \cite{ALM} in which a crucial point was, indeed, the introduction of a perturbation approach based only on fine estimates of  Dyson-Phillips expansions which contrasts with the previous existing results, usually based on a resolvent approach. \smallskip

Before entering the details of the present paper, let us spend few words recalling briefly some properties of the class of Banach spaces we shall deal  with in this paper (which coincide with the one investigated in \cite{ALM}). In all this paper, we shall assume that $\mathcal{E}$ is a real ordered Banach space with a generating positive cone $%
\e_{+}$ (i.e. $\e=\e_{+}-\e_{+}$) on which the norm is
additive, i.e.
$$\left\| u+v\right\| =\left\| u\right\| +\left\| v\right\|\qquad \qquad u,v \in \e_{+}.$$
The additivity of the norm  implies
that the norm is monotone, i.e.
$$0 \leq u \leq v \Longrightarrow \|u\| \leq \|v\|.$$
In particular, the cone $\e_+$ is normal \cite[Proposition
1.2.1]{batty} and any bounded monotone sequence of $\e_+$ is convergent. Moreover, there exists a
linear positive functional $\mathbf{\Phi} $ on $\e$ which coincides with the norm on
the positive cone (see e.g. \cite[p. 30]{Da1}), i.e.
\begin{equation}\label{psi}
\mathbf{\Phi} \in \e^\star_+,\qquad \la\mathbf{\Phi}, u \ra
=\|u\|,\qquad u \in \e_+\end{equation}
We note also that by a Baire category argument there exists a constant $\gamma >0$ such that each $u \in \e$ has a decomposition
\begin{equation}\label{decomp} u=u_1-u_2 \text{ where } u_i \in \e_+ \text{ and } \|u_i\|\leq \gamma \|u\| \quad (i=1,2);\end{equation} i.e. the positive cone
$\X_+$ is {non-flat}, see \cite[Proposition 19.1]{depagter}.

In such abstract state space $\e$, the natural question addressed by honesty theory of subtochastic semigroups  is the following: let the generator  $\mathcal{A}$    of a positive semigroup $(\mathcal{U}(t))_{t\geq0}$ in $\e$ be given with $\|\mathcal{U}(t)u\| \leq \|u\|$ for any $u\in \e_+$ and let a nonnegative operator  $\mathcal{B}\::\:\D(\mathcal{B}) \subset \e \to \e$  with $\D(\mathcal{A}) \subset \D(\mathcal{B})$ be such that
\begin{equation}\label{intemu}\la \mathbf{\Psi}, (\mathcal{A}+\mathcal{B})u \ra \leq 0 \qquad \forall u \in \D(\mathcal{A}) \cap \e_+.\end{equation}
It is known (following the approach of a seminal paper of Kato \cite{kato}) that there exists an extension $\mathcal{K}$ of the sum $(\mathcal{A}+\mathcal{B},\D(\mathcal{A}))$ which generates a nonnegative semigroup $(\mathcal{V}(t))_t$ in $\e$. The scope of  honesty theory of substochastic semigroups is then to provide   criteria 
 in order to determine whether $\mathcal{K}$ is the closure of $\mathcal{A}+\mathcal{B}$ or not (notice that, in general, $\mathcal{K}$ may be a proper extension of $\mathcal{A}+\mathcal{B}$). We emphasize here that, for formally conservative equations, i.e. whenever inequality \eqref{intemu} is an equality,  the property $\mathcal{K}=\overline{\mathcal{A}+\mathcal{B}}$ is necessary and sufficient  to assert that
the corresponding semigroup is norm-preserving on the positive cone, i.e.
$$\mathcal{K}=\overline{\mathcal{A}+\mathcal{B}}\quad \Longleftrightarrow \quad \|\mathcal{ V}(t) u \|= \| u \|  \text{ for any } u \in \e_+ \text{ and any } t \geq 0.$$
Whenever \eqref{intemu} is 'only' an inequality, the concept of honesty of the semigroup $(\mathcal{V}(t))_{t \geq 0}$ is more involved and we refer to \cite{mkvo, ALM} for further details (see also Section \ref{reminder}). We notice only here that the above inequality \eqref{intemu} implies that
\begin{equation}\label{intBu} \int_0^t \|\mathcal{BU}(t)u\|\d t \leq \|u\| -\|\mathcal{U}(t)u\| \qquad \forall  u \in \D(\mathcal{A}) \cap \e_+\,,\:t \geq 0\end{equation}
and an identity in \eqref{intemu} would yield an identity in \eqref{intBu}. \medskip

As already said, in the present paper, we wish to extend this kind of problematic to the more delicate case of evolution families $(\U(t,s))_{t \geq s \geq 0}$. We recall here the following definition
(see, e.g., \cite{chicone, gulisashvili}):
\begin{defi}\label{evfam}
A family $(\U(t,s))_{t\geq s}$ of bounded operators  in $\e$
 is called an {\it evolutionary family} if the
following conditions are
fulfilled:
\begin{enumerate}[(i)]
\item  for each $u \in \e$ the function $(t,s) \mapsto
\U(t,s)u$ is continuous for
$t\geq s$;
\item $\U(t,s) =\U(t,r)\U(r,s)$, $\: \U(t,t) =\mathbf{I}$, $t\geq r \geq s$;
\item $\|\U(t,s)\|\leq Ce^{\beta (t-s)}$, $t\geq s$ for some
constants $C$, $\beta >0$.\end{enumerate}
\end{defi}

The starting point for our analysis will be then the following set of Assumptions that will hold for all the rest of the paper:
\begin{hyp}\label{hypo} Let  $(\U(t,s))_{t\geq s}$ be a given evolution family and let $(\mathbf{B}(t))_{t \geq 0}$ be a family of unbounded operators in $\e$. We assume the following
\begin{enumerate}[(i)\:]
\item $\U(t,s)u \in \e_+$ for any $u \in \e_+$ and $t \geq s \geq
0$;
\item $\|\U(t,s)u\|_\e \leq \|u\|_\e$ for any $u \in \e_+$ and any $t \geq
s \geq 0$;
\item for any $t\geq 0$, $\mathbf{B}(t)u \in \e_+$ for any $u \in \D(\mathbf{B}(t)) \cap \e_+$;
\item for any $u \in \e$ and any $s \geq 0$ one has
\begin{equation}\label{crucial}
\U(t,s)u \in \D(\mathbf{B}(t)) \qquad \text{ for a. e. } t \geq 0,\end{equation}  and the mapping $t \in [s,\infty)\mapsto \mathbf{B}(t)\U(t,s)u$ is measurable. Moreover, for any $t \geq s \geq 0$ and any $u \in  \e_+$  the following holds:
\begin{equation}\label{estimalpha}\int_s^{t}
\left\|\mathbf{B}(\tau)\U(\tau,s)u\right\|_\e \d \tau \leq
\|u\|_\e-\|\U(t,s)u\|_\e.
\end{equation}
\end{enumerate}
Whenever the inequality in \eqref{estimalpha} is an \textbf{\emph{equality}}, we shall say that we are in the \textbf{\emph{formally conservative case}}.
\end{hyp}
Clearly, points (i) and (ii) mean  that the unperturbed evolution family  $(\U(t,s))_{t \geq s}$ is substochastic. The family of unbounded nonnegative operators $(\mathbf{B}(t))_{t \geq 0}$ will then be seen as perturbation of $(\U(t,s))_{t\geq s}$ where inequality \eqref{estimalpha} will play the role, in this non-autonomous framework, of the above \eqref{intBu}. Assumption (iv) is crucial in our analysis and will be met in the various applications to linear kinetic equations we have in mind.

In Section 2, borrowing ideas to Kato's seminal paper \cite{kato}, we can prove (see Theorem \ref{generevol}) that, under Assumptions \ref{hypo}, there exists an evolution family $\vt$ in $\e$ such that
\begin{equation}\label{equ38-intro}
\V(t,s)u =\U(t,s)u + \int_s^t \V(t,r)\mathbf{B}(r)\U(r,s)u \d r,
\qquad \forall u \in \e.
\end{equation}
Moreover, the evolution family $\vt$ is substochastic, i.e. $\left\|\V(t,s)u\right\|_\e \leq \|u\|_\e$ for any $u \in \e_+$. Notice that, in order to the above identity to make sense, our assumption that $\U(r,s)u \in \D(\mathbf{B}(r))$ for a. e. $r \geq s$ plays a crucial role. The proof of such an existence result is \textit{constructive} and the evolution family $\vt$ satisfying \eqref{equ38-intro} is built as  a generalization of the Dyson--Phillips
expansion series:
$$\V(t,s)u=\sum_{n=0}^\infty \V_n(t,s)u \qquad \forall u \in \e, \:t \geq s \geq 0$$
with $\V_0(t,s)=\U(t,s)$ and
\begin{equation}\label{dyson-intro}
\V_{n+1}(t,s)u=\int_s^t \V_n(t,r)\mathbf{B}(r)\U(r,s)u \d r,\qquad
 t \geq s \geq 0,\:\:u \in
\e.\end{equation}
The above series converges in $\e$ and one can prove that it provides an evolution family satisfying the variation of constant formula \eqref{equ38-intro}.

Notice that perturbation theory of strongly continuous evolution families is already a well-developed research area which followed somehow the paths of the perturbation theory of $C_0$-semigoups. In particular, we refer the reader to the seminal work  \cite{miya} where a non-autonomous Miyadera perturbation theory has been developed. Our results in the present paper \emph{do not resort to such a Miyadera theory}, however some of the ideas implemented in \cite{miya} will play an important role in our analysis. More precisely, in \cite{miya} (see also \cite{liske, monniaux, rabiger}), the perturbation theory for the evolution family is based upon known perturbation results for $C_0$-semigroups via the following well-known feature of evolution family \cite{chicone}: to a given evolution family, say $(\U(t,s))_{t \geq s \geq 0}$, one can associate  a $C_0$-semigroup $(\T_0(t))_{t \geq 0}$ acting on the \textit{lifted space}
$$\X=L^1(\mathbb{R}_+,\mathcal{E})$$
endowed with its natural norm  $\|\cdot\|_{\X}$  given by
 $\|f\|_{\X}=\int_0^\infty \|f(s)\|_\e \d s$ for any $f \in \X$ (see Section \ref{sec:B}) for more details). The semigroup $(\T_0(t))_{t \geq 0}$ is given by:
$$\left[{\T}_0(t)f\right](r)=\chi_+(r-t) {\U}(r,r-t)f(r-t) \qquad r \geq 0$$
where $\chi_+(\cdot)$ stands for the indicator function of the set $\mathbb{R}^+$.

While the approach of \cite{miya} consists in applying, in the space $\X$, the now classical Miyadera perturbation theorem \cite{arloban} to that semigroup $(\T_0(t))_{t \geq 0}$, our approach in the present paper is different since the construction of $\vt$ is \textit{non-autonomous} in nature: we do not work in the space $\X$ to construct $\vt$. In this sense, our approach  is essentially self-contained and no use of the results of \cite{miya} is necessary for our analysis.\smallskip

However, we will borrow some of the  ideas of \cite{miya}  to deal with the honesty theory of $\vt$, exploiting in particular the results we obtained in \cite{ALM} for perturbed semigroup \textit{in the space} $\X$.  The fundamental observation is that the lifted space $\X$ is also an \emph{abstract state space} whose norm  $\|\cdot\|_\X$  can be
uniquely extended to a positive bounded linear functional
$\mathbf{\Psi}$ satisfying
\begin{equation*}
\mathbf{\Psi} \in \X^\star_+,\qquad \la\mathbf{\Psi},f \ra_\X
=\|f\|_\X,\qquad f \in \X_+.\end{equation*}
Now, to the perturbed evolution family $\vt$ corresponds a unique $C_0$-semigroup $\Tt$ in $\X$ and both the semigroups $(\T_0(t))_{t\geq 0}$ and $\Tt$ are substochastic semigroups in $\X$ and one can determine the link between the two semigroups and their respective generators $(\Z,\D(\Z))$ and $(\G,\D(\G))$.

Precisely, we construct in Section \ref{sec:B} and borrowing some ideas of \cite{voigt}, a perturbation operator  $\widehat{\fB}\::\:\D(\widehat{\fB}) \subset \X \to \X$ which is nonnegative in $\X$, such that $\D(\widehat{\fB}) \supset \D(\Z)$ and satisfying
$$\la \mathbf{\Psi}, \left(\Z+\widehat{\fB}\right)f\ra_\X \leq 0 \qquad \forall f \in \D(\Z) \cap \X_+.$$
In other words, the honesty theory for $C_0$-semigroup developed in \cite{ALM} applies in the space $\X$ and $\G$ is an extension of the sum $(\Z+\widehat{\fB},\D(\Z))$. It is moreover possible to prove that the semigroup $\Tt$ is exactly the one obtained by a direct application, in $\X$, of Kato's theorem \cite{ALM}. Notice that building the perturbation $\widehat{\fB}$ and recognizing that $\Tt$ coincide with the semigroup obtained thanks to Kato's theorem  will require several additional measurability assumptions, besides Assumptions \ref{hypo}, that will be introduced progressively in the text.

Clearly, the great interest of working in the lifted space $\X$ is that it allows us to exploit the results established in \cite{ALM} about honesty of $C_0$-semigroups and we shall adopt the following definition for the honesty of $\vt$:
\begin{defi} The evolution family $\vt$ is honest in $\e$ if and only if the associated semigroup $\Tt$ is honest in $\X$.
\end{defi}

As already said, several characterizations of the honesty of $\Tt$ have been established in our recent contribution \cite{ALM} but it is very challenging to be able to express such characterizations in explicit terms involving only the known objects in $\e$ that are $\mathbf{B}(t)$ and the Dyson-Phillips iterated $\V_n(t,s)$.

This is the object of the core of our paper, developed in Section \ref{sec:sec4}. The novelty of our approach lies in a constant interaction between results established in the space $\X$ and results typical to $\e$ while the previous existing approach \cite{miya,liske} essentially works in the lifted space $\X$ to deduce property in $\e$. Let us explain with more details the results of Section \ref{sec:sec4}: we first provide, under some additional assumptions on $(\mathbf{B}(t))_{t \geq 0}$, a practical interpretation of the perturbation $\widehat{\fB}$ as simply the \emph{multiplication operator} in $\X$ by $\mathbf{B}(\cdot)$. Such an additional assumption, see Assumption \ref{hypo2}, is clearly restrictive but turns out to be very well-suited to the case of non-integral operators that we have in mind for applications.

Second, we construct a second family of Dyson-Phillips iterated $(\overline{\V}_n(t,s))_{n \in \mathbb{N}}$ $(t \geq s\geq 0)$ which is such that, for any $u \in \e_+$ and any $n \in \mathbb{N}$,
\begin{equation}\label{eq:BV0intro}\overline{{\V}}_n(t,s)u \in \D({\mathbf{B}}(t)) \quad \text{ for almost every } t \geq s.\end{equation}
Finally, we show that this second family of iterated actually coincide with the original one \eqref{dyson-intro}, i.e. $\overline{\V}_n(t,s)=\V_n(t,s)$ for any $t \geq s\geq 0$ and any $n \in \mathbb{N}$. This allows to prove the main result of the paper (Theorem \ref{theo:main?}) that asserts that the evolution family $\vt$ is honest if and only if the following holds for any $u \in \e_+$
 $$\lim_{n \to \infty} \int_s^t \left\|\mathbf{B}(\t)\V_n(\t,s)u\right\|_\e \d\t=0 \qquad \forall t \geq s.$$
Notice that in  the formally conservative case for simplicity, one easily sees that the evolution family $\vt$ is conservative if and only if it is honest. \medskip

In Section 5, we deal with applications to kinetic theory in arbitrary geometry under the assumption that the cross-sections are space homogeneous. Non-autonomous kinetic equation have already been investigated in \cite{vdmee} underlying in particular the influence of \textit{time-dependent} forcing term. Our approach is different here and we are dealing rather with \textit{time-dependent} collision operators. The
main technical point is to show how Assumptions \ref{hypo} are
satisfied in the context of kinetic theory. In addition, as in the autonomous case
\cite{musAfrika}, we show also how the honesty of space homogeneous kinetic equations of the form
\begin{equation*}
\partial_t f(v,t)+\Sigma(v,t) f(v,t) =\int_{V} b(t,v,v')f(v',t)\d\mu(v') \qquad v  \in V
\end{equation*}
implies the honesty of full (transport) kinetic equations
\begin{equation*}\label{kineticintro}
\left\{\begin{split}
\partial_t f(x,v,t)+v \cdot \nabla_x f(x,v,t) &+ \Sigma(v,t) f(x,v,t)\\
&=\int_{V} b(t,v,v')f(x,v',t)\d\mu(v') \qquad x \in \mathbf{\Omega}, \qquad v \in V\\
f(x,v,t)&=0 \qquad  \text{ if } (x,v) \in \mathbf{\Gamma}_-
\end{split}\right.
\end{equation*}
where $\mathbf{\Omega} \subset \R^d$ is a open subset, $\mu$ is a Borel measure over $\R^d$ with support $V$ (we refer to Section \ref{sec:inhomo} for details on the various terms in the above equation). We investigate this problem in the state space $\e=L^1(\mathbf{\Omega}\times V,\d x\otimes\d\mu(v))$. Here, as in \cite{musAfrika}, we deal with \textit{singular} and \textit{subcritical} kinetic equation: namely, the kinetic equation is singular in the sense that the collision operator
$$f \in L^1(V,\d\mu(v)) \longmapsto \int_V b(t,v,v')f(v')\d\mu(v') \in L^1(V,\d\mu(v))$$
is an \textit{unbounded operator} while the subcritical assumption refers to
$$\int_V b(t,v,v')\d\mu(v') \leq \Sigma(t,v) \qquad \text{ for a. e. } (t,v) \in \mathbb{R}^+ \times V.$$
We finally  provide a practical criterion of honesty of space homogeneous kinetic equations in terms of existence
of a detailed balance condition, inspired by the results of \cite{musAfrika}.\medskip

In section 6, we revisit briefly some of the results of \cite{LB} for non-autonomous fragmentation equation
\begin{equation*}
\partial_t u(t,x)=-a(t,x)u(t,x) + \int_{x}^\infty a(t,y)b(t,x,y)u(t,y)\d y  \qquad x >0, \:t > s\end{equation*}
under some reasonable assumptions on $a(t,\cdot)$ and $b(t,\cdot,\cdot)$ (see Section \ref{sec:frag} for details). We show how the abstract theory developed in the first four sections of the present paper do fit to such a problem and allow to recover in a very simple way several results of \cite{LB}.

\section{Perturbation theory of substochastic evolution families}

 As it  is the case for autonomous
problem (i.e. for  $C_0$-semigroup), it is possible to obtain an
abstract generation result for evolution families.
 Under the above set of Assumptions, one can construct an  evolution family $\vt$ in $\e$.
\begin{theo}\label{generevol}
Assume that Assumptions \ref{hypo} hold true. Then, there exists an evolution family $\vt$ in $\e$ such that
\begin{equation}\label{equ38}
\V(t,s)u =\U(t,s)u + \int_s^t \V(t,r)\mathbf{B}(r)\U(r,s)u \d r,
\qquad \forall u \in \e.
\end{equation}
Moreover, $\left\|\V(t,s)u\right\|_\e \leq \|u\|_\e$ for any $u \in \e_+$ while $\left\|\V(t,s)u\right\|_\e \leq 2\gamma\|u  \|_\e$ for any $u \in \e.$
\end{theo}

The proof of the above result consists actually in constructing the
evolution family $\vt$ as a generalization of the Dyson--Phillips
expansion series. Precisely, define for any $n \in \mathbb{N}$:
\begin{equation}\label{V-n}\begin{split}
\V_0(t,s)&=\U(t,s),  \qquad t \geq s \geq 0, \qquad \qquad \text{
and }\\
\V_{n+1}(t,s)u&=\int_s^t \V_n(t,r)\mathbf{B}(r)\U(r,s)u \d r,\qquad
 t \geq s \geq 0,\:\:u \in
\e.\end{split}\end{equation} One has
\begin{propo}\label{luisanotes2}
For any $n \geq 0$ and any $t \geq s \geq 0$,  $\V_n(t,s) \in
\mathscr{B}(\mathcal{E})$ with the following properties:
\begin{enumerate}[(1)\:]
\item $\V_n(t,s)u \in \e_+$ for any $u  \in \e_+;$
\item $\sum_{k=0}^n \|\V_k(t,s)u\|_\e \leq 2\gamma\|u\|_\e$ for any  $u
\in \e$ while
\begin{equation}\label{vn++}\sum_{k=0}^n \|\V_k(t,s)u\|_\e \leq \|u\|_\e \quad \text{ for any } \quad  \:u
\in \e_+.\end{equation}
\item For any $0 \leq s \leq r \leq t$:
$$\V_n(t,s)u=\sum_{k=0}^n\V_k(t,r)\V_{n-k}(r,s)u,\qquad \forall u \in \e.$$
\item For any $n \geq 1$, $\lim_{h \to 0} \V_n(t+h,t)u=0$ for
any $u \in \e,$ $t \geq 0.$
\item For any $s \geq 0$, the mapping $ [s, \infty) \ni t \mapsto
\V_n(t,s)$ is strongly continuous.
\item For any $r \geq 0$ and any $t > r$,  the mapping
$ s \in (r,t) \longmapsto \V_n(t,s)$ is strongly continuous.
\end{enumerate}
\end{propo}
\begin{proof} We
prove all the above points \textit{(1)--(6)} by induction. All these
points are clearly satisfied for $n=0$ (except clearly point \textit{(4)}). Assume now they hold for any
$k \leq n$ for some given $n \in \mathbb{N}$ and let us prove they
still hold true for $k=n+1.$ In all the sequel, we assume $t \geq s \geq 0$ to be fixed.


\textit{ (1)} Let $u \in \e_+$ be given. Since $\V_n(t,s)$ is a bounded positive linear operator in $\e$, $\V_{n+1}(t,s)u$ is well-defined by Eq. \eqref{V-n} and it is clear from Assumptions \ref{hypo} \textit{(iii)} that $\V_{n+1}(t,s)u \in \e_+$.

\textit{ (2)} Let us first assume  $u \in
 \e_+$ to be given.  The induction assumption implies that
$$\|\V_n(t,r)v\|_\e \leq \|v\|_\e -\sum_{k=0}^{n-1} \|\V_k(t,r)v\|_\e$$ for any $t \geq r \geq 0$ and any $v \in \e_+$. Applying this to $v=\mathbf{B}(r)\U(r,s)u \in \e_+$ and using \eqref{estimalpha}, we
get:
\begin{equation*}\begin{split}
\|\V_{n+1}(t,s)u\|_\e &\leq \int_s^t \|\mathbf{B}(r)\U(r,s)u\|_\e \,
\d r -\sum_{k=0}^{n-1}\int_s^t \|\V_k(t,r)\mathbf{B}(r)\U(r,s)u\|_\e
\, \d r\\ &\leq \|u\|_\e -\|\U(t,s)u\|_\e -\sum_{k=0}^{n-1}\int_s^t
\|\V_k(t,r)\mathbf{B}(r)\U(r,s)u\|_\e \, \d r.
\end{split}
\end{equation*}
Now,  one has
\begin{equation*}\begin{split} \int_s^t
\|\V_k(t,r)&\mathbf{B}(r)\U(r,s)u\|_\e \, \d r =\int_s^t \la
\mathbf{\Phi},
\V_k(t,r)\mathbf{B}(r)\U(r,s)u  \ra_\e\d r\\
&=\la \mathbf{\Phi}, \int_s^t \V_k(t,r)\mathbf{B}(r)\U(r,s)u \d
r\ra_\e=\la \mathbf{\Phi}, \V_{k+1}(t,s)u\ra_\e\\
&=\left\| \V_{k+1}(t,s)u\right\|_\e
\end{split}\end{equation*} from which we deduce that
\begin{equation}\label{Vn1-Utu}
\|\V_{n+1}(t,s)u\|_\e \leq \|u\|_\e -
\sum_{k=0}^{n}\| \V_k(t,s)u\|_\e   \qquad \forall u \in
 \e_+, \qquad t \geq s \geq 0.\end{equation}
This proves \eqref{vn++} and we deduce from \eqref{decomp} that $\|\V_{n+1}(t,s)u\|_\e \leq 2\gamma\|u\|_\e$ for any $u \in
\e$ and any $t \geq s \geq 0.$

\textit{(3)} Let $t \geq r \geq s \geq 0$ be given and let $u \in
\e$ be fixed. Using the induction assumption, one has
\begin{multline}\label{sumn+1}
\sum_{k=0}^{n+1}\V_k(t,r)\V_{n+1-k}(r,s)u =\U(t,r)\V_{n+1}(r,s)u+
\V_{n+1}(t,r)\U(r,s)u\\+\sum_{k=1}^{n}\V_k(t,r)\V_{n+1-k}(r,s)u.\end{multline}
Now, $\U(t,r)\V_{n+1}(r,s)u$ is equal to
\begin{equation*}\begin{split}
\U(t,r)&\int_{s}^r\V_n(r,\t) \mathbf{B}(\t)\U(\t,s)u\d\t =\int_s^r\U(t,r)\V_n(r,\t)\mathbf{B}(\t)\U(\t,s)u\d\t\\
&=\int_s^r\V_n(t,\t)\mathbf{B}(\t)\U(\t,s)u\d \t
-\sum_{j=1}^n\int_s^r\V_j(t,r)\V_{n-j}(r,\t)\mathbf{B}(\t)\U(\t,s)u\d\t
\end{split}\end{equation*}
where we used the fact that the induction formula is true for any $k
\leq n$. In the same way,
\begin{equation*}\begin{split}\V_{n+1}(t,r)\U(r,s)u&=\int_r^t\V_n(t,\sigma)\mathbf{B}(\sigma)\U(\sigma,r)\U(r,s)u\d\sigma\\
&=\int_r^t\V_n(t,\sigma)\mathbf{B}(\sigma)\U(\sigma,s)u\d\sigma\end{split}\end{equation*}
from which we deduce that
\begin{multline}\label{sumn+11}
\U(t,r)\V_{n+1}(r,s)u+\V_{n+1}(t,r)\U(r,s)u=\int_s^t\V_n(t,\t)\mathbf{B}(\t)\U(\t,s)u\d
\t\\-\sum_{j=1}^n\int_s^r\V_j(t,r)\V_{n-j}(r,\t)\mathbf{B}(\t)\U(\t,s)u\d\t.\end{multline}
Now, $$\V_{n+1-k}(r,s)u =
\int_s^r\V_{n-k}(r,\t)\mathbf{B}(\t)\U(\t,s)u\d\t$$ so that
\begin{equation*}
\sum_{k=1}^{n}\V_k(t,r)\V_{n+1-k}(r,s)u=\sum_{k=1}^n\int_s^r\V_k(t,r)\V_{n-k}(r,\t)\mathbf{B}(\t)\U(\t,s)u\d\t,\end{equation*}
which, from \eqref{sumn+1} and \eqref{sumn+11} proves that
$$\sum_{k=0}^{n+1}\V_k(t,r)\V_{n+1-k}(r,s)u=\int_s^t\V_n(t,\t)\mathbf{B}(\t)\U(\t,s)u\d
\t=\V_{n+1}(t,s)u$$ and the result holds true for $n+1$.

\textit{(4)} Let $n \geq 0$ and $u \in \e_+$ be fixed and let $h \geq 0$. Then, according to \eqref{Vn1-Utu},
$$\|\V_{n+1}(t+h,t)u\| \leq \|u\| - \|\U(t+h,t)u\|$$
from which we easily deduce that
$$\lim_{h \to 0^+} \|\V_{n+1}(t+h,t)u\|=0 \qquad \forall u \in \e_+$$
and the result extends to any $u \in \e$ by linearity.

\textit{(5)} 
Given $t \geq 0$ and $h >0$, one has from \textit{(3)}
\begin{multline}\label{0.6}
\V_{n+1}(t+h,s)u-\V_{n+1}(t,s)u =\sum_{k=1}^{n+1}\V_k(t+h,t)\V_{n+1-k}(t,s)u +\\
\left(\U(t+h,t)-\mathbf{I}\right)\V_{n+1}(t,s)u.
\end{multline}
Therefore, since \textit{(4)} holds for any $1 \leq k \leq n+1$ and, from the strong continuity of $\U(\cdot,t)$ we clearly get from \eqref{0.6} that
$$\lim_{h \to 0} \left\|\V_{n+1}(t+h,s)u-\V_{n+1}(t,s)u\right\|=0.$$
One proves the result similarly for $h <0$. Therefore, the mapping $t \mapsto \V_{n+1}(t,s)u$ is continuous over $[s,\infty).$

\textit{(6)} Let $t > s > r$ and $h \in \mathbb{R}$ such that $s+h \in (r,t).$ Assume first that $h >0$. Since
$$\V_n(t,s) = \sum_{k=0}^{n}\V_k(t,s+h)\V_{n-k}(s+h,s)$$
one has
$$\V_n(t,s+h) - \V_n(t,s) = \V_n(t,s+h)(\mathbf{I} - \U(s+h,s)) + \sum_{k=0}^{n-1}\V_k(t,s+h)\V_{n-k}(s+h,s)$$
and clearly $\V_n(t,s+h) - \V_n(t,s)$ strongly converges to zero as $h \to 0^+$. One proceeds in the same way for $h < 0$ and this proves point \textit{(6)}.\end{proof}


With this result in hands, one can prove Theorem \ref{generevol}.

\begin{proof}[Proof of Theorem \ref{generevol}] As it is the case for Kato's Theorem (see \cite{ALM}), the proof consists only in defining the
evolution family $\vt$ as the strong limit of $\sum_{k=0}^n
\V_k(t,s).$ Precisely, for any fixed  $t \geq s \geq 0$ and any $u
\in \e_+$, the sequence $(\sum_{k=0}^n \V_k(t,s)u)_n$ is
nondecreasing and bounded according to Proposition
\ref{luisanotes2}, \textit{(2)}. Therefore, one can define
\begin{equation}\label{limitV}\V(t,s)u=\sum_{n=0}^\infty \V_n(t,s)u=\lim_{N \to \infty} \sum_{n=0}^N \V_n(t,s)u \qquad \text{ for any } \qquad u \in \e_+,\end{equation}
and $\|\V(t,s)u\|_\e \leq \|u\|_\e,$ $\forall u \in\e_+.$  Let us now prove \eqref{equ38}. Let $u \in
\e $ and $t \geq s \geq 0$ be fixed. By
construction, for any $n \in \mathbb{N}$:
\begin{equation}\label{sumn}
\sum_{k=0}^n \V_k(t,s)u=\U(t,s)u+\int_0^t \sum_{k=1}^n
\V_{k-1}(t,r)\mathbf{B}(r)\U(r,s)u\d r.\end{equation} Now, according
to Prop. \ref{luisanotes2} \textit{(2)},
$$\left\|\sum_{k=1}^n \V_{k-1}(t,r)\mathbf{B}(r)\U(r,s)u\right\|_\e \leq 2\gamma
\|\mathbf{B}(r)\U(r,s)u\|_\e \qquad \forall 0 \leq r \leq t$$ and,
since the mapping $r \in (s,t) \mapsto \|\mathbf{B}(r)\U(r,s)u\|_\e$ is
integrable, one can use \eqref{limitV} together with the dominated convergence theorem to  get
$$\lim_{n \to \infty}\int_s^t \sum_{k=1}^n
\V_{k-1}(t,r)\mathbf{B}(r)\U(r,s)u\d r=\int_s^t
\V(t,r)\mathbf{B}(r)\U(r,s)u\d r$$
which, combined with \eqref{sumn} yields \eqref{equ38}. Moreover,
thanks to Prop. \ref{luisanotes2}, \textit{(3)},  one gets easily
that
\begin{equation}\label{propagation} \V(t+s)u=\V(t,r)\V(r,s) u, \qquad \forall  t \geq r \geq s \geq 0\,, \qquad u \in
\e.\end{equation} Let us now prove that
\begin{enumerate}[i)\:]
\item for any $s \geq 0$, the mapping $ [s, \infty) \ni t \mapsto
\V(t,s)$ is strongly continuous; \item for any $t \geq 0$, the
mapping $(0,t] \ni s \mapsto \V(t,s)$ is strongly continuous at
$s=t$. \end{enumerate}

Let us prove i) for any fixed $s \geq 0$, i.e.
$$\lim_{h \to 0^+}\|\V(t+h,s)u-\V(t,s)u\|_\e=0, \qquad \forall t \geq
s, \qquad u \in \e.$$

From \eqref{propagation}, as in the proof of Proposition \ref{luisanotes2} \textit{(4)}, it is enough to prove that
\begin{equation}\label{continV}\lim_{h \to 0^+}\|\V(s+h,s)u-\V(s,s)u\|_\e=0, \qquad
\forall u \in \e.\end{equation} Let $u \in \e $ be fixed. One has
\begin{equation*}\begin{split}
\|\V(s+h,s)u-\V(s,s)u\|_\e&=\|\V(s+h,s)u-u\|_\e \\  &\leq
\|\V(s+h,s)u-\U(s+h,s)u\|_\e+\|\U(s+h,s)u-u \|_\e.
\end{split}\end{equation*}
According to \eqref{equ38},
\begin{equation*}
\|\V(s+h,s)u-\U(s+h,s)u\|_\e \leq 2\gamma\int_s^{s+h}
\left\|\mathbf{B}(r)\U(r,s)u\right\|_\e \d r
\end{equation*}
which clearly converges to $0$ as $h \to 0^+.$ Since $\lim_{h \to 0^+
}\|\U(s+h,s)u-u\|_\e=0$, this proves \eqref{continV}. We prove ii)
in the same way.  Now, according to \cite[Lemma 1.1]{liske}, all the above properties are enough to assert that $\vt$ is
a {\it evolutionary family} in $\e$.  \end{proof}

\section{A semigroup approach}\label{sec:B}

It is by now well-documented that non-autonomous Cauchy
problems can be handled with through an appropriate semigroup formulation (see \cite[Chapter 3]{chicone} for a thorough discussion on this topic). Such an approach consists actually in associating to an
exponentially bounded evolution family $(\U(t,s))_{t \geq s}$ on
$\mathcal{E}$ a $C_0$-semigroup  on the vector-valued space
$$\X=L^1(\mathbb{R}_+,\mathcal{E})$$
endowed with its natural norm  $\|\cdot\|_{\X}$  given by
$$\|f\|_{\X}=\int_0^\infty \|f(s)\|_\e \d s, \qquad f \in \X.$$
It is clear that $\|\cdot\|_\X$ is additive on the positive cone
$$\X_+=\{f \in \X, \:\:f(s) \in \e_+ \text{ a. e. } s \geq 0\}$$
and, as in the Introduction, the norm on $\X_+$ can be
uniquely extended to a positive bounded linear functional
$\mathbf{\Psi}$ satisfying
\begin{equation*}
\mathbf{\Psi} \in \X^\star_+,\qquad \la\mathbf{\Psi},f \ra_\X
=\|f\|_\X,\qquad f \in \X_+.\end{equation*} Clearly, $\mathbf{\Psi}$
is related to $\mathbf{\Phi}$ by:
\begin{equation}\label{psi2}
\la \mathbf{\Psi}, f \ra_\X=\int_0^\infty \la \mathbf{\Phi}, f(s)
\ra_\e \d s, \qquad \qquad \forall f \in \X_+.\end{equation}
One can define the operator family $(\T_0(t))_{t \geq 0}$
on $\X$:
\begin{equation}\label{lien}
[\T_0(t)f](s)=\begin{cases} \U(s,s-t)f(s-t),\qquad  &s \geq t \geq 0,\\
 0 \qquad \qquad &s \geq 0, s-t \notin
\mathbb{R}_+.\end{cases}\end{equation} It is easily verified that
$(\T_0(t))_{t \geq 0}$ is a $C_0$-semigroup in $\X$ that we shall
call the {\it evolutionary semigroup} associated to the evolution
family $\U$ and we denote its generator by $(\Z,\D(\Z))$. Notice that, given $f \in \X_+$, since $f(s) \in \e_+$ for almost every $s \geq 0$, one deduces from Assumptions \ref{hypo} \textit{(ii)} that
\begin{multline*}
\|\T_0(t)f\|_\X = \int_0^\infty \|[\T_0(t)f](s)\|_\e \d s = \int_t^\infty \|\U(s,s-t)f(s-t)\|_\e \d s\\
\leq \int_t^\infty \|f(s-t)\|_\e \d s =\int_0^\infty \|f(r)\|_\e \d r=\|f\|_\X.
\end{multline*}
Using the terminology of \cite{ALM}, the semigroup $(\T_0(t))_{t \geq 0}$ is substochastic. Equivalently,
$$\la \mathbf{\Psi}, \T_0(t)f\ra_\X \leq \la \mathbf{\Psi}, f\ra_\X \qquad \forall t \geq 0, \: f \in \X_+.$$
In particular, if $f \in \D(\Z) \cap \X_+$, one gets
$$\la \mathbf{\Psi}, \Z f \ra_\X = \lim_{t \to 0^+} t^{-1} \la \mathbf{\Psi}, \T_0(t)f - f \ra_\X \leq 0.$$
In the same way, to the evolution family $\vt$ in $\e$ constructed in Theorem \ref{generevol}, one can associate a \textit{substochastic} $C_0$-semigroup $\Tt$ in $\X$ defined by
\begin{equation}\label{lineTt}
[\T (t)f](s)=\begin{cases} \V(s,s-t)f(s-t),\qquad  &s \geq t \geq 0,\\
 0 \qquad \qquad &s \geq 0, s-t \notin
\mathbb{R}_+, \qquad f \in \X.\end{cases}\end{equation}
Its generator will be denoted by $(\G,\D(\G))$ and
 \begin{equation}\label{substoG}
\la \mathbf{\Psi}, \G f \ra_\X \leq 0 \qquad \forall f \in \D(\G) \cap \X_+.\end{equation}
\subsection{Construction of an additive perturbation $\widehat{\fB}$} We want to understand in this section the link between $\G$ and $\Z$. In particular, we aim to prove that $\G$ is the (minimal) extension of some additive perturbation of $\Z$, of course related to the family $(\mathbf{B}(t))_{t\geq 0}$. To do so, a first natural attempt would be to introduce the 'multiplication' operator in $\X$ :
\begin{equation}\begin{split}\label{deffB}\fB\: f=\mathbf{B}(\cdot)f(\cdot)\qquad &\text{ with } \\
\phantom{++++} \D(\fB)&=\{f \in
\X\;,\,f(s) \in \D(\mathbf{B}(s))\, \text{ a.e. } s >0,
\;\mathbf{B}(\cdot)f(\cdot) \in \X\}.\end{split}\end{equation}
 To do so, we need additional assumptions on the family $(\mathbf{B}(t))_{t\geq 0}$:
\begin{hyp}\label{hypo1} Let $\Delta=\{(t,s) \;;\;0 \leq s \leq t < \infty\}$ 
In addition to Assumptions \ref{hypo}, we assume that, for any $f \in \X$, the mapping
$$(t,s) \in \Delta \longmapsto \mathbf{B}(t)\U(t,s)f(s) \in \e$$
is well-defined for almost $(t,s) \in \Delta$ and is measurable.
\end{hyp}

\begin{nb} For any $f \in \X$, the mapping
$$(t,s) \in \Delta \longmapsto \mathbf{B}(t)\U(t,s)f(s) \in \e$$
turns out be integrable. Indeed, for any $T > 0$, set $\Delta_T=\{(s,r)\;;\; 0 \leq r \leq s \leq T\}$ and assumes first $f(s) \in \e_+$ for almost every $s \geq 0$. Then,  by Fubini's theorem
\begin{multline*}
\int_{\Delta_T} \|\mathbf{B}(t)\U(t,s)f(s)\|_\e \d s \d t=\int_0^T \d t\int_0^t \|\mathbf{B}(t)\U(t,s)f(s)\|_\e\d s\\
=\int_0^T \d s \int_s^T \|\mathbf{B}(t)\U(t,s)f(s)\|_\e\d t
\leq \int_0^T \|f(s)\|_\e \d s
\end{multline*}
according to Assumptions \ref{hypo} (iv). Letting then $T \to \infty$, we get that
$$\int_{\Delta} \|\mathbf{B}(t)\U(t,s)f(s)\|_\e \d s \d t \leq \|f\|_\X \qquad \forall f \in \X_+$$
for which we deduce easily the result for general $f \in \X.$
\end{nb}

\textbf{\emph{In all the sequel, we shall assume that Assumptions \ref{hypo1} are in force.}} Introduce the  following definition:
\begin{defi} For any $\l >0$ and any $f \in \X$, let $\fB_\l f\;:\;s \geq 0 \mapsto [\fB_\l f](s)$ be given by
\begin{equation}\label{represfBl}[\fB_\l f](s)=\int_0^s \exp(-\l(s-\tau))\mathbf{B}(s)\U(s,\tau)f(\tau)\d \tau \qquad \text{ for a.e. } s \geq 0.\end{equation}
\end{defi}
Then, one has the following
\begin{propo} For any $\l >0$ and any $f \in \X$,  $\fB_\l f \in \X$ with
$$\|\fB_\l f \|_\X \leq \|f\|_\X \qquad \forall f \in \X_+.$$
In particular, $\fB_\l\,:\,f \in \X \mapsto \fB_\l f$ is a bounded operator in $\X$. Moreover, for any $f \in \X$ and any $\l >0$, the following holds
\begin{multline}\label{equa:1fB}
\left[(\l-\G)^{-1}f\right](t)=\left[(\l-\Z)^{-1}f\right](t) \\
 + \int_0^t \exp(-\l(t-r))\V(t,r)\left[\fB_\l f\right](r)\d r \qquad \text{ for a.e. }t \geq 0.
\end{multline}
In other words,
\begin{equation}\label{lgBl}(\l-\G)^{-1}\fB_\l f=(\l-\G)^{-1}f -(\l-\Z)^{-1}f \qquad \forall f \in \X, \qquad \forall \l >0.\end{equation}
\end{propo}
\begin{proof} The first part of the proposition is easy to prove using \eqref{estimalpha} and we focus on the proof of \eqref{equa:1fB}. Let $f \in \X$ be fixed. Since $\G$ is the generator of $(\T_0(t))_{t \geq 0}$, one has for any $\l >0$,
\begin{equation*}\begin{split}
\left[(\l-\Z)^{-1}f\right](t)&=\left[\int_0^\infty \exp(-\l s)\T_0(s)f\d s\right](t)\\
&=\int_0^\infty \exp(-\l s)\left[\T_0(s)f\right](t)\d s \qquad \text{ for a.e. } t \geq 0.\end{split}\end{equation*}

Notice that, to prove the above second identity, we can invoke \cite[Lemma 3.2]{miya} since the mapping $s \in \mathbb{R}^+ \mapsto \exp(-\l s)\T_0(s)f \in \X$ is integrable. Using again \cite[Lemma 3.2]{miya}, since the mapping  $(t,s) \in \mathbb{R}^+ \times \mathbb{R}^+ \mapsto \exp(-\l s)\chi_+(t-s)\U(t,t-s)f(t-s) \in \e$ is measurable and, for any $s \geq 0,$ one has $[\exp(-\l s)\T_0(s)f](t)=\exp(-\l s)\chi_+(t-s)\U(t,t-s)f(t-s)$ for almost every $t \geq 0$,  we get
\begin{equation}\label{eq:l-Z}\left[(\l-\Z)^{-1}f\right](t)=\int_0^t \exp(-\l(t-s))\U(t,s)f(s)\d s \qquad \text{ for a.e. } t \geq 0.\end{equation}
In the same way,
\begin{equation}\label{eq:l-G}\left[(\l-\G)^{-1}f\right](t)=\int_0^t \exp(-\l(t-s))\V(t,s)f(s)\d s \qquad \text{ for a.e. } t \geq 0.\end{equation}
Now, for a given $s \in (0,t),$ applying Duhamel formula \eqref{equ38} to $u=f(s)$, we get
\begin{multline*}\int_0^t \exp(-\l(t-s))\V(t,s)f(s)\d s=\int_0^t \exp(-\l(t-s))\U(t,s)f(s)\d s \\
+ \int_0^t \exp(-\l(t-s))\d s \int_s^t \V(t,r)\mathbf{B}(r)\U(r,s)f(s)\d r\end{multline*}
which, thanks to Assumptions \ref{hypo1} and according to Fubini's theorem, yields
\begin{multline*}\int_0^t \exp(-\l(t-s))\V(t,s)f(s)\d s=\int_0^t \exp(-\l(t-s))\U(t,s)f(s)\d s \\
+ \int_0^t \d r \int_0^s \exp(-\l(t-s))\V(t,r)\mathbf{B}(r)\U(r,s)f(s)\d s.\end{multline*}
Since $\V(t,r) \in \mathscr{B}(\e)$ for any $t\geq r \geq 0$, we easily get that
\begin{multline*}\int_0^t \exp(-\l(t-s))\V(t,s)f(s)\d s
=\int_0^t \exp(-\l(t-s))\U(t,s)f(s)\d s \\
+ \int_0^t \exp(-\l(t-r))\V(t,r) \d r \int_0^s \exp(-\l(r-s))\mathbf{B}(r)\U(r,s)f(s)\d s\end{multline*}
which, combined with \eqref{eq:l-G}, \eqref{eq:l-Z} and \eqref{represfBl} yields \eqref{equa:1fB}.
\end{proof}

An important consequence of the above Proposition is the following which is inspired by \cite{voigt}
\begin{theo}\label{construcB} One has $\D(\Z) \subset \D(\G)$. Let $\widehat{\fB}$ with domain $\D(\widehat{\fB})=\D(\Z)$ be defined as
$$\widehat{\fB}f=\G f -\Z f \qquad \forall f \in \D(\Z).$$
Then, $\widehat{\fB}$ is $\Z$-bounded and for any $f \in \D(\Z)$ one has
\begin{equation}\label{rela}\widehat{\fB}f=\lim_{\l  \to\infty}\l  \fB_\l  f.\end{equation}
In particular, $\widehat{\fB}f \in \X_+$ for any $f \in \D(\Z) \cap \X_+$ and
\begin{equation}\label{z+B}\la \mathbf{\Psi}, \left(\Z + \widehat{\fB}\right) f\ra_\X \leq 0 \qquad \forall f \in \D(\Z) \cap \X_+.\end{equation}
Finally, the following Duhamel's identity holds
\begin{equation}\label{duha1}
\T(t)f=\T_0(t)f+\int_0^t \T(t-s)\widehat{\fB}\T_0(s)f\d s \qquad \forall f \in \D(\Z), \:t\geq 0.
\end{equation}
\begin{proof} The fact that $\D(\Z) \subset \D(\G)$ is a direct consequence of \eqref{lgBl}. With such a property, the above definition of the unbounded operator $\widehat{\fB}$ makes sense. Notice that, with such a definition, $(\G,\D(\G))$ can be seen now as an extension of $(\Z + \widehat{\fB},\D(\Z))$. Let now $f \in \X$ and $\l >0$, setting $g=(\l-\Z)^{-1}f$, applying the operator $(\l-\G)$ to the left of Eq. \eqref{lgBl} yields
$$\fB_\l f=f-(\l-\G)g=(\l-\Z)g-(\l-\G)g=(\G-\Z)g=\widehat{\fB}g$$
since $g \in \D(\Z).$ This exactly means that
\begin{equation}\label{def:fBl}\fB_\l=\widehat{\fB}\left(\l-\Z\right)^{-1} \in \mathscr{B}(\X) \qquad \forall \l >0.\end{equation}
In particular,  $\widehat{\fB}$ is $\Z$-bounded. Now, as a generator of a $C_0$-semigroup, the operator $(\Z,\D(\Z))$ satisfies
$$\lim_{\l  \to \infty}\l  (\l-\Z)^{-1}f=f \qquad \text{ and } \qquad \lim_{\l  \to \infty} \l  \Z(\l-\Z)^{-1}f=\Z f \qquad \forall
f \in \D(\Z)$$
(see e.g. \cite[Lemma II.3.4]{engel}). Using the fact that $\widehat{\fB}$ is $\Z$-bounded, this proves \eqref{rela} and, since $\fB_\l$ is non-negative, we get that $\widehat{\fB}$ is a non-negative operator. Finally,  \eqref{z+B} follows easily from \eqref{substoG} since $\G$ is an extension of $(\Z+\widehat{\fB},\D(\Z))$. It remains now to prove Duhamel's identity \eqref{duha1}. To do so, one notices that, according to
\eqref{lgBl} and \eqref{def:fBl},
\begin{equation}\label{identGBZ}
(\lambda - \G)^{-1}f = (\lambda - \Z)^{-1}f + (\lambda - \G)^{-1}\widehat{\fB}(\lambda - \Z)^{-1}f \qquad \forall f \in \X,\;\l >0.\end{equation}
Now, for $f \in \D(\Z)$, $\T_0(s)f \in \D(\Z)$ for any $s \geq 0$ and, since $\widehat{\fB}$ is $\Z$-bounded, one has
\begin{equation}\label{fbZla}\widehat{\fB}(\l-\Z)^{-1}f=\widehat{\fB}\int_0^\infty \exp(-\l s) \T_0(s)f\d s=\int_0^\infty \exp(-\l s)\widehat{\fB}\T_0(s)f\d s.\end{equation}
Thus, thanks to Fubini's theorem
\begin{equation*}\begin{split}
(\l-\G)^{-1}\widehat{\fB}(\l-\Z)^{-1}f&=\int_0^\infty \exp(-\l \t)\T(\t)\left(\int_0^\infty \exp(-\l s)\widehat{\fB}\T_0(s)f\d s\right)\d\t\\
&=\int_0^\infty  \left(\int_0^\infty \exp(-\l (s+\t))\T(\t)\widehat{\fB}\T_0(s)f\d \t\right)\d s\\
&=\int_0^\infty \left(\int_s^\infty \exp(-\l t) \T(t-s)\widehat{\fB}\T_0(s)f\d t\right)\d s\end{split}\end{equation*}
with the change of variable $\t \to t=s+\t.$ Using again Fubini's theorem yields then
$$(\lambda - \G)^{-1}\widehat{\fB}(\lambda - \Z)^{-1}f= \int_0^{\infty}\exp(- \lambda t) \left(\int_0^t \T(t-s)\widehat{\fB}\T_0(s)f\d s\right)\d t$$
and one sees that \eqref{identGBZ} reads
\begin{multline*}
\int_0^\infty \exp(-\l t)\T(t)f \d t =\int_0^\infty \exp(-\l t)\T_0(t)f\d t\\
+ \int_0^{\infty}\exp(- \lambda t) \left(\int_0^t \T(t-s)\widehat{\fB}\T_0(s)f\d s\right)\d t \qquad f \in \D(\Z),\,\l >0\end{multline*}
which is exactly \eqref{duha1}.\end{proof}
\end{theo}
\begin{nb}\label{remarkVoigt}
If we define, as in \cite{miya} the set $\mathfrak{D}$ as the subspace of
$\X$ made of  all linear combinations  of elements of the form: $$t >0 \longmapsto
\varphi(t)\U(t,s)u\,\qquad \text{ where } \,s >0\,,\, u\in
\e$$ and $\varphi \in \mathscr{C}_c^1(\mathbb{R})$ with $\mathrm{supp}(\varphi) \in
[s,\infty)$, then one can prove (see Lemma \ref{lemmeAvarphis} in Appendix A) that $\mathfrak{D} \subset \D(\Z)$,  $\mathfrak{D} \subset \D(\fB)$ and $\widehat{\fB}\vert_{\mathfrak{D}}=\fB\vert_{\mathfrak{D}}$. Thus, $\widehat{\fB}$ is an extension of $\fB\vert_{\mathfrak{D}}$. Moreover, if $f \in \mathfrak{D}$, Duhamel's identity \eqref{duha1} holds with $\fB$ instead of $\widehat{\fB}.$
\end{nb}

From the general \textit{substochastic semigroup theory} as developed in \cite{ALM} (more precisely, according to Kato's theorem \cite[Theorem 2.1 \& Theorem 2.3]{ALM}), one deduces from \eqref{z+B} that there exists  an extension of $(\Z+\widehat{\fB},\D(\Z))$ that generates a \emph{minimal} substochastic \com in $\X$. Our goal is to prove that such an extension is exactly $\G$ and the minimal semigroup is exactly $\Tt$. To do so, let us introduce the following
\begin{defi}\label{defi:Tnover1}
For any $f \in \X$, $t \geq 0$ and $n \in \mathbb{N}$, define
$$\left[{\T}_n(t)f\right](r)=\chi_+(r-t) {\V}_n(r,r-t)f(r-t) \qquad r \geq 0$$
\end{defi}
One has the following
\begin{lemme} For any $n \in \mathbb{N}$ and any $t \geq 0$, $\T_n(t) \in \mathscr{B}(\X).$ Moreover, for any $\l >0$
\begin{equation}\label{laplaceTn}
\int_0^\infty \exp(-\l t)\T_n(t)f \d t=(\l-\Z)^{-1}
\left[\widehat{\fB}(\l-\Z)^{-1}\right]^n f \qquad \forall f \in \X.
\end{equation}
Finally, for any $f \in \D(\Z)$ and any $n \geq 1$

\begin{equation}\label{Tn+1}\T_n(t)f=\int_0^t \T_{n-1}(t-s)\widehat{\fB}\T_0(s)f\d s\qquad t \geq 0.\end{equation}
\end{lemme}
\begin{proof} The fact that $\T_n(t)$ is a bounded operator for any $n \in \mathbb{N}$ and any $t \geq 0$ is easily seen. Let us now prove \eqref{laplaceTn} by induction. It is clear that, for $n=0$, \eqref{laplaceTn} holds true. Assume now it holds for some fixed $n \in \mathbb{N}$. Let now $f \in \X$. Considering that $(t,\t) \in \mathbb{R}^+ \times \mathbb{R}^+ \mapsto \exp(-\l \t)\chi_+(t-\t)\V_{n+1}(t,t-\t)f(t-\t) \e$ is measurable, we deduce from \cite[Lemma 3.2]{miya} that the following holds for any $t \geq 0$
\begin{equation*}\begin{split}
\left[\int_0^{\infty}\exp(- \l \t)\T_{n+1}(\t)f \d \t\right](t)&=\int_0^\infty \exp(-\l \t)\chi_+(t-\t)\V_{n+1}(t,t-\t)f(t-\t)\d \t\\
  &=\int_0^{t}\exp(- \lambda(t - s))\V_{n+1}(t,s)f(s)\d s.\end{split}\end{equation*}
For some fixed $s \in (0,t)$, using \eqref{V-n} with $u=f(s)$ one obtains
$$\left[\int_0^{\infty}\exp(- \l \t)\T_{n+1}(\t)f \d \t\right](t)= \int_0^{t}\left(\int_s^t \exp(- \lambda(t - s))\V_{n}(t,r)\mathbf{B}(r)\U(r,s)f(s) \d r\right)\d s.$$
Considering that the mapping $(r,s) \in \Delta \mapsto \V_n(t,r)\mathbf{B}(r)\U(r,s)f(s)\in \e$ is measurable and that $\V_n(t,r)$ is a bounded operator for any $t,r \geq 0$, we obtain thanks to Fubini's theorem that
\begin{equation*}\begin{split}
\bigg[\int_0^{\infty}\exp(- \l \t)&\T_{n+1}(\t)f \d \t\bigg](t) = \int_0^{t} \left(\int_0^r \exp(- \lambda(t - s))\V_n(t,r)\mathbf{B}(r)\U(r,s)f(s) \d s\right)\d r \\
&=\int_0^t \exp(-\l (t-r))\V_n(t,r) \left(\int_0^r \exp(- \lambda(r - s)) \mathbf{B}(r)\U(r,s)f(s) \d s\right)\d r.\end{split}\end{equation*}
Now, from \eqref{represfBl} one gets
\begin{equation*}
\left[\int_0^{\infty}\exp(- \l \t)\T_{n+1}(\t)f \d \t\right](t) =\int_0^{t}\exp(- \lambda(t - r))\V_{n}(t,r)[ \fB_\lambda f](r)\d r.\end{equation*}
Arguing as before, one recognizes that
\begin{equation*}
\int_0^{\infty}\exp(- \l \t) \T_{n+1}(\t)f \d \t =\int_0^{\infty}\exp(- \lambda \t)\T_{n}(\t)\fB_\lambda f\d \t.\end{equation*}
The induction hypothesis finally yields
$$\int_0^{\infty}\exp(- \l \t)  \T_{n+1}(\t)f \d \t=(\l-\Z)^{-1}
\left[\widehat{\fB}(\l-\Z)^{-1}\right]^n \fB_\l f$$
which exactly means that \eqref{laplaceTn} holds for $n+1$ since $\fB_\l=\widehat{\fB}(\l-\Z)^{-1}.$ Let us now prove \eqref{Tn+1}. For $n \in \mathbb{N}$, $n \geq 1$ and $f \in \D(\Z)$, one sets
$$\mathcal{J}_n(t):=\int_0^t \T_{n-1}(t-s)\widehat{\fB}\T_0(s)f\d s.$$
Notice that, $\T_0(s)f \in \D(\Z)$ so $\widehat{\fB}\T_0(s)f$ is meaningful for any $s \geq 0.$ Now, according to \cite[Theorem 2.3 \& Eq. (2.13)]{ALM}, one knows that
$$\int_0^\infty \exp(-\l t)\mathcal{J}_n(t)\d t=(\l-\Z)^{-1}\left[\widehat{\fB}(\l-\Z)^{-1}\right]^{n}f \qquad \forall \l >0.$$
Using \eqref{laplaceTn}, one sees that $\mathcal{J}_n(t)$ and $\T_n(t)f$ have the same Laplace transform which shows \eqref{Tn+1}.
\end{proof}
\begin{nb}\label{importantNB} Notice that, on the basis of formula \eqref{Tn+1}, one can reproduce the argument of  \cite[Proposition 4.1 (i) \& (ii)]{ALM} and prove that the following hold:
\begin{enumerate}
\item For any $f \in \D(\Z)$, the mapping $t \in (0,\infty) \longmapsto \T_n(t)f$ is
continuously differentiable with
\begin{equation}\label{point1}\dfrac{\d}{\d t}\T_n(t)f=\T_n(t)\Z f  +
\T_{n-1}(t) \widehat{\fB} f.\end{equation}
\item For any $f \in \D(\Z)$,  $\T_n(t)f \in \D(\Z)$, the mapping $t \in [0,\infty) \longmapsto \Z\T_n(t)f$ is continuous and
\begin{equation}\label{point2}\Z\T_n(t)f=\T_n(t)\Z f +\T_{n-1}(t)\widehat{\fB} f-\widehat{\fB} \T_{n-1}(t) f.\end{equation}
\end{enumerate}
Indeed, the proof of \cite[Proposition 4.1 (i) \& (ii)]{ALM} relies only on the relation \eqref{Tn+1} and not on the fact that the sequence of iterated is obtained via Kato's theorem.
\end{nb}
Notice that, besides  \eqref{Tn+1}, one can state another equivalent induction relation between $\T_n(t)$ and $\T_{n+1}(t)$, namely
\begin{lemme}\label{lem:Tn+1bis}
For any $f \in \D(\Z)$ and any $n \geq 1$ it holds
\begin{equation}\label{Tn+1bis}\T_n(t)f=\int_0^t \T_0(t-s)\widehat{\fB}\T_{n-1}(s)f\d s\qquad t \geq 0.\end{equation}
\end{lemme}
\begin{proof} Let $f \in \D(\Z)$ and $n \geq 1$ be fixed. Define $g_n(t)=\T_n(t)f$, $t \geq 0.$ One knows from point (2) of Remark \ref{importantNB} that $g_{n}(t)  \in \D(\Z)$ for any $t \geq 0$. Moreover, combining \eqref{point1} and \eqref{point2}
one checks easily that $g_{n}(\cdot)$ is continuously differentiable.
Since moreover $ g_{n}(0)=0$, one sees that the mapping $t \geq 0 \mapsto g_{n}(t) \in \X$ is a classical solution to the Cauchy problem
\begin{equation*}\begin{cases}
\dfrac{\d}{\d t}g_{n}(t)&=\Z g_{n}(t)+ \widehat{\fB}\T_{n-1}(t)f \qquad t > 0\\
g_{n}(0)&=0 \end{cases}\end{equation*}
and, since $\Z$ is the generator of the \com $(\T_0(t))_{t \geq 0}$, one gets that
$$g_{n}(t)=\int_0^t \T_0(t-s)\widehat{\fB}\T_{n-1}(s)f\d s= \int_0^t \T_0(s)\widehat{\fB}\T_{n-1}(t-s)f\d s \forall  t \geq 0$$
which is exactly \eqref{Tn+1bis}.\end{proof}
With this in hands, one can finally prove that the semigroup $\Tt$ is exactly the one constructed thanks to Kato's theorem \cite{ALM}. Namely, one has the following:
\begin{theo}\label{kato} For any $\lambda > 0$,
the resolvent of $\G$ is given by
\begin{equation}\label{resolG}
(\l-\G)^{-1}f=\lim_{n \to \infty} (\l-\Z)^{-1}\sum_{k=0}^n
\left[\widehat{\fB}(\l-\Z)^{-1}\right]^k f , \qquad f \in \X.\end{equation}
The semigroup $\Tt$ is the smallest substochastic \com whose generator is an
extension of $(\Z+\widehat{\fB},\D(\Z))$ and
$$\T(t)f=\sum_{n=0}^\infty \T_n(t)f \qquad\forall t \geq 0, \:f \in\X.$$
\end{theo}
\begin{proof} For $f \in \X_+$ one has, for any $N \geq 0$, $\sum_{n=0}^N \T_n(t)f \in \X_+$ and
\begin{multline*}
\left\|\sum_{n=0}^N \T_n(t)f\right\|_{\X}=\int_0^\infty \chi_{+}(r-t)\left\|\sum_{n=0}^N \V_n(r,r-t)f(r-t)\right\|_\e \d r\\
\leq \int_t^\infty \sum_{n=0}^N \left\|\V_n(r,r-t)f(r-t)\right\|_\e \d r.
\end{multline*}
For some given $r \geq t \geq 0$, since $f(r-t) \in \e_+$, one can apply \eqref{vn++} to get
$$\sum_{n=0}^N \left\|\V_n(r,r-t)f(r-t)\right\|_\e \leq \|f(r-t)\|_\e$$
and
$$\left\|\sum_{n=0}^N \T_n(t)f\right\|_\X \leq \int_t^\infty \|f(r-t)\|_\e \d r=\int_0^\infty \|f(s)\|_\e \d s=\|f\|_X.$$
Therefore, the sequence $\left(\sum_{n=0}^N \T_n(t)f\right)_{N}$ is a nonnegative and bounded sequence of $\X$. It converges towards some limit that we denote $\widehat{\T}(t)f$. Since $\V(t,s)u=\sum_{n=0}^\infty \V_n(t,s)u$ for any $u \in \e$ and any $(t,s) \in \Delta$, it is not difficult to deduce from Definition \ref{defi:Tnover} that, for any $f \in \X_+$, the limit $\widehat{\T}(t)f$ has to coincide with $\T(t)f$, i.e. 
$$\T(t)f=\sum_{n=0}^{\infty} \T_n(t)f\qquad \forall t \geq 0.$$
Now, since for any $\l >0$ one has
$$(\l-\G)^{-1}f=\int_0^\infty \exp(-\l t)\T(t)f\d t=\lim_{N \to \infty}\sum_{n=0}^N \int_0^\infty \exp(-\l t)\T_n(t)f\d t$$
one immediately gets \eqref{resolG} from \eqref{laplaceTn} for any $f\in \X_+$. Clearly, the identity extends to general $f \in \X$ by linearity. Now, with Theorem \ref{construcB} and the subsequent Remark, one sees that \eqref{resolG} is enough to conclude that the semigroup $\Tt$ is exactly the one obtained applying the classical substochastic semigroup theory in $\X$ (see \cite[Theorem 2.1]{ALM}) and the other properties follow.
\end{proof}

At this stage, as just explained in the above proof, one sees finally that the semigroup $\Tt$ is exactly the semigroup obtained by applying Kato's theorem \cite[Theorem 2.1]{ALM} to the operators $\Z$ and its perturbation $\widehat{\fB}$. In particular, we shall be able, in the next section to exploit the recent results about the honesty theory for perturbed substochastic semigroups and extend them to build an honesty theory for perturbed evolution families.

\subsection{ Reminder on the honesty theory of $C_0$-semigroups}\label{reminder}

We recall here some of the results we obtained recently for perturbed substochastic semigroups in abstract Banach spaces. We refer the reader to \cite{ALM} for further details. The honesty of the semigroup $\Tt$ is defined through several useful functionals. Precisely, let
 $$\cc_0\:\::\:\:f \in \D(\G) \mapsto \cc_0(f)=-\la \mathbf{\Psi},\G
f \ra_\X \in \mathbb{R}.$$
and let $\cc$ denote its  restriction
to $\D(\Z)$, i.e.
$$\cc\:\::\:\:f \in \D(\Z) \mapsto \cc (f)=-\la \mathbf{\Psi},\Z f + \widehat{\fB}
f \ra_\X \in \mathbb{R}.$$
Then, one has the following result which allows us to define a third functional $\widehat{\cc}$:
\begin{propo}(\cite[Prop. 4.5]{ALM})\label{hatccdef} For any $f \in \D(\G)$, there exists
\begin{equation}\label{hatcc}\lim_{t \to 0^+}\dfrac{1}{t}\sum_{n=0}^\infty \cc\left(\int_0^t \T_n(s)f\d s\right)=:\widehat{\cc}(f).\end{equation} Furthermore, for $f \in \D(\G)_+$, $\widehat{\cc}(f) \leq \cc_0(f) \leq \|\G f\|.$
\end{propo}
With this in hands, one defines the concept of honest trajectory \cite[Definition 4.7]{ALM} (see also \cite[Section 3]{ALM} for an equivalent notion of honesty):
\begin{defi}\label{defi:hon}
Let $f \in \X_+$ be given. Then, the trajectory
$(\T(t)f)_{t \geq 0}$ is said to be \textit{\textbf{honest}} if and
only if
$$\|\T(t)f\|_\X=\| f\|_\X-\widehat{\cc}\bigg(\int_0^t \T(s)f \d
s\bigg),\quad \text{ for any }   t\geq 0.$$  The whole
$C_0$-semigroup $\Tt$ will be said to be honest if all trajectories
are honest.
\end{defi}
\begin{nb} Notice that, in the \emph{formally conservative case}, by virtue of Proposition \ref{propo:equivRela}, $\cc(f)=0$ for any $f \in \D(\Z)$ so that the trajectory
$(\T(t)f)_{t \geq 0}$ is honest if and
only $\|\T(t)f\|_\X=\| f\|_\X$ for any $t \geq 0.$
Thus the whole semigroup is honest if and only if it is conservative.
\end{nb}

We recall, at this stage, the definition of the honesty for the evolution family $\vt$:

\begin{defi}\label{defi:honest} The evolution family $\vt$ is said to be honest in $\e$ whenever the $C_0$-semigroup $\Tt$ in $\X$, given by \eqref{lineTt}, is honest in $\X$.
\end{defi}
\begin{nb} We will provide, in Section \ref{sec:hovt}, that it is also possible to give a definition of honest trajectory $(\V(t,s)u)_{t \geq s}$ in $\e$ for any $u \in \e_+$.
\end{nb}

Necessary and sufficient conditions of honesty are provided in \cite{ALM} in terms of the perturbation $\widehat{\fB}$ (see \cite[Theorem 4.8]{ALM}):
\begin{theo}\label{equivalencemild} Given $f \in \X_+$, the following statements are equivalent
\begin{enumerate}
\item the trajectory $(\T(t)f)_{t \geq 0}$ is  honest;\\

\item $\int_0^t \T(s)f\d s \in \D(\overline{\Z+\widehat{\fB}})$ for any $t >0$;\\

\item the set $\left(\widehat{\fB}\int_0^t \T_n(s)f\d s\right)_n$ is relatively weakly compact in $\X$ for any $t >0$;\\

\item $\lim_{n \to \infty}\left\|\widehat{\fB}\int_0^t \T_n(s)f\d s\right\|_\X=0$ for any $t >0$.

\end{enumerate}
\end{theo}

Clearly, these sufficient conditions are given in $\X$ and involve  mathematical objects defined in the lifted space $\X$. This is not satisfactory when dealing with the honesty of the evolution family $\vt$ in $\e$. The purpose of the next section is to provide an explicit criterion for $\vt$ involving only mathematical objects defined in $\e$, namely $(\mathbf{B}(t))_{t \geq 0}$ and the sequence of iterated $(\V_n(t,s))_n$. The first technical difficulty will be to give a practical interpretation (in terms of $\mathbf{B}(\cdot)$) of the perturbation $\widehat{\fB}$.

\section{Criterion for the honesty of $\vt$}\label{sec:sec4}

\subsection{Practical interpretation of $\widehat{\fB}$}\label{sec:practical} In this paragraph, we provide a practical representation formula for the operator $(\widehat{\fB},\D(\Z))$. The following results use several fine properties of both the operators $\Z$ and $\widehat{\fB}$, which for the convenience of the reader, have been postponed in an Appendix (see Appendix A).

To obtain a practical characterization of the operator $\widehat{\fB}$, we need the following assumption
\begin{hyp}\label{hypB} Assume that the family  $(\mathbf{B}(t))_{t\geq 0}$ is such that, for any $t \geq 0,$ $\mathbf{B}(t)$ is \emph{'increasingly closed'} in the sense that: if $(u_n)_n \subset \D(\mathbf{B}(t)) \cap \e_+$ is increasing with $\lim_n u_n=u \in \e_+,$ and $\sup_n \|\mathbf{B}(t)u_n\|_\e < \infty$ then $u \in \D(\mathbf{B}(t))$ and
$$\lim_n \mathbf{B}(t)u_n=\mathbf{B}(t)u.$$
\end{hyp}
\begin{nb} Clearly, if $\mathbf{B}(t)$ is a closed operator then the above  holds true. Notice however that assuming that $\mathbf{B}(t)$ for any $t\geq 0$ is a closed operator is a too strong hypothesis for the applications we have in mind. On the contrary, Assumption \ref{hypB} will be satisfied in the practical situations we will consider in the second part of the paper. It holds true in particular whenever $\mathbf{B}(t)$ is a suitable integral operator with nonnegative kernel on some $L^1$-space.
\end{nb}
With this additional assumption, the operator $\widehat{\fB}$ has the following interpretation:
\begin{propo}\label{closable} Assume that Assumptions \ref{hypo1} and \ref{hypB} are in force and let $f \in \D(\Z)$ be given. Then $f(t)\in \D( {\mathbf{B}}(t))$ for almost every $t \geq 0$ and
$$\left[\widehat{\fB}f\right](t)= {\mathbf{B}}(t) f(t) \qquad \text{ for almost every } t \geq 0.$$
In other words, $\widehat{\fB}f=\fB f$ for any $f \in \D(\Z)$, i.e. $\D(\Z) \subset \D(\fB)$ and $\fB\vert_{\D(\Z)}=\widehat{\fB}$.
\end{propo}

\begin{nb} The reader may legitimately ask why we introduced a new operator $\widehat{\fB}$ if, at the end, $\widehat{\fB}$ is only the multiplication operator $\fB$. Here are two answers to this natural question. First, it appears very difficult, even under Assumption \ref{hypB} to prove \emph{directly} that $\D(\fB)\supset \D(\Z)$ and that $\fB f =\G f -\Z f$ for any $f \in \D(\Z)$. Second, the conclusion of Prop. \ref{closable} may change drastically if Assumption \ref{hypB} does not hold.  Indeed, if instead of Assumption \ref{hypB}, one assumes (which would sound more natural) that $\mathbf{B}(t)$ is closable for any $t\geq 0$ with closure in $\e$ given by $\overline{\mathbf{B}}(t)$, then one can prove following the approach of \cite{liske} that, for any $f \in \D(\Z)$:
$$\left[\widehat{\fB}f\right](t)=\overline{\mathbf{B}}(t)f(t) \qquad \text{ for a.e.} t \geq 0.$$
We leave the details to the reader.\end{nb}

\begin{proof}[\textit{Proof of Proposition \ref{closable}:}] Let $g \in \X_+$ be given and let $\l >0.$ Since $\lim_{t \to 0^+}\|\exp(-\l t)\T_0(t)g-g\|_\X=0$ and $\fB_\l$ is a bounded operator, one has
$$\lim_{t \to 0^+}\|\fB_\l \exp(-\l t)\T_0(t)g-\fB_\l g\|_\X=0.$$
In particular, for any decreasing positive sequence $(t_n)_n$, with $\lim_n t_n=0^+$, there exists a subsequence (still denoted by $(t_n)_n$) such that
$$\lim_{n \to \infty}\left[\fB_\l \exp(-\l t_n)\T_0(t_n)g\right](s)=[\fB_\l g](s) \qquad \text{ a.e. } s \geq 0$$
where, of course, the convergence is meant in $\e$. According to Lemma \ref{lemBlexp}, we can choose the sequence $(t_n)_n$ such that
$$\left[\fB_\l \exp(-\l t_n)\T_0(t_n)g\right](s)=\mathbf{B}(s)\int_0^{s-t_n}\exp(-\l (s-\t))\U(s,\tau)g(\tau)\d \tau \qquad \text{for a. e. } s \geq t_n.$$
This means that the following convergence holds in $\e$:
$$\lim_{n \to \infty} \mathbf{B}(s)\int_0^{s-t_n}\exp(-\l(s-\tau))\U(s,\tau)g(\tau)\d\tau=[\fB_\l g](s) \qquad \text{a.e. } s \geq 0.$$
For such a choice of $s \geq 0$, let
$$u_n=\int_0^{s-t_n}\exp(-\l(s-\tau))\U(s,\tau)g(\tau)\d\tau.$$
Then, by virtue of \eqref{eq:l-Z} and since $g(\tau) \in \e_+$ for a.e. $\tau \geq 0$, $(u_n)_n$ is an increasing sequence of $\e_+$ which converges to
$$u=\left[(\l-\Z)^{-1}g\right](s).$$
Since moreover $(\mathbf{B}(s)u_n)_n$ converges (towards $[\fB_\l g](s)$), it is a bounded sequence of $\e$ and, according to Assumption \ref{hypB}, $u \in \D(\mathbf{B}(s))$ and
$$\lim_n \mathbf{B}(s)u_n=\mathbf{B}(s)u.$$
This reads $\left[(\l-\Z)^{-1}g\right](s) \in \D(\mathbf{B}(s))$ with
$$[\fB_\l g](s)=\mathbf{B}(s)u=\mathbf{B}(s)\left[(\l-\Z)^{-1}g\right](s).$$
Then, for $g \in \X$, writing $g=g_1-g_2$ with $g_i \in \X_+$ $i=1,2$, we get that
$$[\fB_\l g](s)=\mathbf{B}(s)\left[(\l-\Z)^{-1}g\right](s) \qquad \text{ for a.e.} s \geq 0.$$
Let now $f \in \D(\Z)$. There is $\l >0$ and $g \in \X$ such that $f=(\l-\Z)^{-1}g$ and the above identity reads now $[\widehat{\fB}f](s)=\mathbf{B}(s)f(s)$ for a.e. $s \geq 0$ which proves the Proposition.
\end{proof}

Using the above properties of $\widehat{\fB}$ and results obtained in Appendix A, we can prove the following:
\begin{propo}\label{prop:intDB}  Assume that Assumptions \ref{hypo1} and \ref{hypB} are in force. For any $f \in \X$, one has
$$\int_0^s \U(s,r)f(r)\d r \in \D( {\mathbf{B}}(s)) \quad \text{ for almost every } s \geq 0$$
with
$$ {\mathbf{B}}(s)\int_0^s \U(s,r)f(r)\d r=\int_0^s \mathbf{B}(s)\U(s,r)f(r)\d r \quad \text{ for almost every } s \geq 0.$$
\end{propo}
\begin{proof}  The proof is a simple consequence of the  Lemma \ref{lem:AppA2} given in Appendix.
\end{proof}

With the above interpretation of $\widehat{\fB}$ one can make more precise condition \eqref{rela} in Theorem \ref{construcB}:
\begin{propo}\label{propo:equivRela}  Assume that Assumptions \ref{hypo1} and \ref{hypB} are in force. Then the following are equivalent
\begin{enumerate}
\item $\la \mathbf{\Psi}\,,\,\left(\Z+\widehat{\fB}\right)f \ra_\X \leq 0$ for any $f \in \D(\Z)_+$;\\

\item $\displaystyle \int_s^t \|\mathbf{B}(\t)\U(\t,s)u\|_\e \d\t \leq \|u\|_\e -\|\U(t,s)u\|_\e$ for any $u \in \e_+$ and any $t > s \geq 0.$
\end{enumerate}
Moreover, in the \emph{formally conservative case} both are identities.
\end{propo}
\begin{proof}  Recall that, according to Lemma \ref{lem:AppA3} in the Appendix, for any $f \in \X_+$ and  any $t >0$, it holds
\begin{equation}\label{identBgt}\la \mathbf{\Psi}\,,\widehat{\fB}\int_0^t \T_0(s)f\d s \ra_\X=\int_0^\infty \left(\int_{\t}^{\t+t} \|\mathbf{B}(r)\U(r,\t)f(\t)\|_\e \d r\right)\d \t.\end{equation}
On the other hand, $\Z \int_0^t \T_0(s)f\d s=\T_0(t)f-f$ from which we deduce easily  that
\begin{equation}\label{identZgt}
\la \mathbf{\Psi}\,,\Z \int_0^t \T_0(s)f\d s \ra_\X=\int_0^\infty \|\U(s+t,s)f(s)\|_\e \d s -\int_0^\infty \|f(s)\|_\e \d s.
\end{equation}
Combining \eqref{identBgt} with \eqref{identZgt} we obtain
\begin{multline}\label{identB+Z}
\la \mathbf{\Psi}\,,\,\left(\Z+\widehat{\fB}\right)\int_0^t \T_0(s)f\d s \ra = \int_0^\infty\left(\int_{\t}^{\t+t} \|\mathbf{B}(r)\U(r,\t)f(\t)\|_\e \d r\right)\d \t \\
+ \int_0^\infty \|\U(s+t,s)f(s)\|_\e \d s -\int_0^\infty \|f(s)\|_\e \d s.\end{multline}

Let us now come to the equivalence between $(1)$ and $(2)$. Assume $(1)$ to be satisfied. Then, for any $t >0$ and any $f \in \X_+$ one has from \eqref{identB+Z},
\begin{equation*}
\int_0^\infty\left(\int_{\t}^{\t+t} \|\mathbf{B}(r)\U(r,\t)f(\t)\|_\e \d r\right)\d \t
\leq \int_0^\infty \|f(\t)\|_\e \d \t- \int_0^\infty \|\U(\t+t,\t)f(\t)\|_\e \d \t \end{equation*}
In particular, if $f \in \X_+$ is such that $f(s)=\varphi(s)u$ for any $s \geq 0$ where $u \in \e_+$ and $\varphi$ is a nonnegative continuous function with compact support in $\mathbb{R}^+$, one gets
\begin{equation*}
\int_0^\infty\varphi(\t)\left(\int_{\t}^{\t+t} \|\mathbf{B}(r)\U(r,\t)u\|_\e \d r\right)\d \t
\leq \int_0^\infty \varphi(\t)\left(\|u\|_\e- \|\U(\t+t,\t)u\|_\e \right)\d \t .\end{equation*}
Since the nonnegative function $\varphi \in \mathscr{C}_c(\mathbb{R}_+)$ is arbitrary, one obtains directly that $(2)$ is satisfied. Moreover, an identity in point (1) will clearly result in an identity in point (2).

Assume now that $(2)$ is satisfied and let $f \in \D(\Z)_+$ be fixed. Since $f(\t) \in \e_+$ for a. e. $\t \geq 0$ we deduce from \eqref{identB+Z} that
$$\la \mathbf{\Psi}\,,\,\left(\Z+\widehat{\fB}\right)\int_0^t \T_0(s)f\d s \ra_\X \leq 0 \qquad \forall t \geq 0.$$
Moreover, because $f \in \D(\Z)$ one has
$$\lim_{t \to 0^+}\frac{1}{t}\int_0^t \T_0(s)f\d s= f \quad \text{ and } \quad \lim_{t \to 0^+}\frac{1}{t}\Z\int_0^t \T_0(s)f\d s=\lim_{t \to 0^+} \frac{1}{t}\int_0^t \T_0(s)\Z f\d s=\Z f.$$
Since $\widehat{\fB}$ is $\Z$-bounded, we deduce that $\la \mathbf{\Psi}\,,\,\left(\Z+\widehat{\fB}\right)f \ra_\x \leq 0$ which is point $(1)$.  Again, one easily gets that an identity in point (2) (which is the definition of the formally conservative case) will yield an identity in  (1). \end{proof}

\subsection{About Dyson-Phillips iterated}

In this paragraph, we derive an alternative expression for the Dyson-Phillips iterated introduced in \eqref{V-n}. Namely, let us introduce the family of iterated Dyson-Phillips as follows: for any $u \in \e$ and any $t \geq s \geq 0$, set
$$\overline{\V}_0 (t,s)u=\U(t,s)u$$
and, for any $n \geq 1$, set
\begin{equation}\label{overlV}\overline{\V}_n(t,s)u=\int_s^t \U(t,r){\mathbf{B}}(r)\overline{\V}_{n-1}(r,s)u\d r.\end{equation}
One has  the following
\begin{propo}\label{prop:Propimp}  Assume that Assumptions \ref{hypo} and \ref{hypB} are in force. Let $u \in \e$ and $s \geq 0$ be fixed. Then, for any $n \in \mathbb{N}$, $\overline{\V}_n(t,s)u$ is well-defined with
$$\overline{\V}_n(t,s)u \in \D({\mathbf{B}}(t)) \quad \text{ for almost every } t \geq s$$
with
\begin{equation}\label{BtVn}
{\mathbf{B}}(t) \overline{\V}_n(t,s)u=\int_s^t \mathbf{B}(t)\U(t,r){\mathbf{B}}(r)\overline{\V}_{n-1}(r,s)u\d r.
\end{equation}
Moreover, the mapping $t \in [s,\infty) \mapsto {\mathbf{B}}(t)\overline{\V}_n(t,s)u \in \e$ is integrable.  Finally,
\begin{multline}\label{estimateoverlineV}
 \int_s^t \|{\mathbf{B}}(\t)\overline{\V}_{n+1}(\t,s)u\|_\e \d \t \leq
\int_s^t \|{\mathbf{B}}(\t)\overline{\V}_{n}(\t,s)u\|_\e \d \t - \int_s^t\|\U(t,\t){\mathbf{B}}(\t)\overline{\V}_{n}(\t,s)u\|_\e \d \t\\
\qquad \text{ for any } u \in \e_+, t \geq s,
\end{multline}
with an equality sign in the \emph{formally conservative} case.
\end{propo}
\begin{proof} The proof of the proposition follows by induction. Let us explain why $\overline{\V}_n(t,s)u$ is well-defined for any $n \in \mathbb{N}$. For $n=1$, since $\overline{\V}_0(r,s)u=\U(r,s)u\in \D(\mathbf{B}(r))$ for a.e. $r\geq 0$, there is no problem to define
$$\overline{\V}_1(t,s)u=\int_s^t \U(t,r){\mathbf{B}}(r)\U(r,s)u\d r.$$
Now, with $s \geq 0$ fixed, define $f_0(r)=\chi_+(s-r){\mathbf{B}}(r)\U(r,s)u$, $r \geq 0.$ One has $f_0 \in \X$ according to \eqref{estimalpha} and
$$\overline{\V}_1(t,s)u=\int_0^t \U(t,r)f_0(r)\d r.$$
According to Prop. \ref{prop:intDB}, $\overline{\V}_1(t,s)u \in \D(\mathbf{B}(t))$ for a.e. $t \in [s,\infty)$ and the mapping $t \in [s,\infty) \mapsto {\mathbf{B}}(t)\overline{\V}_1(t,s)u \in \e$ is integrable. Then, $\overline{\V}_2(t,s)u$ is well-defined and, repeating the argument, one gets the first part of the Proposition. Let us now prove \eqref{estimateoverlineV}. Let $u \in \e_+$ and $n \in \mathbb{N}$.  Define $f_n(r)= {\mathbf{B}}(r)\overline{\V}_n(r,s)u \in \e_+$ for a.e. $r \in (s,t)$. Clearly, $f_n \in \X_+$. Now, according to Proposition \ref{prop:intDB}, one has
\begin{equation*}\begin{split}
{\mathbf{B}}(\t)\overline{\V}_{n+1}(\t,s)u&={\mathbf{B}}(\t)\int_s^{\t} \U(\t,r){\mathbf{B}}(r)\overline{\V}_n(r,s)u\d r\\
&=\int_s^{\t} \mathbf{B}(\t)\U(\t,r){\mathbf{B}}(r)\overline{\V}_n(r,s)u \d r \quad \text{ a.e. } \t \in (s,t)\end{split}\end{equation*}
i.e.
$$f_{n+1}(\tau)=\int_s^\tau \mathbf{B}(\tau)\U(\tau,r)f_n(r)\d r \qquad \text{ for a.e.} \tau \in (s,t).$$
In particular $f_{n+1} \in \X_+$ and it follows easily by induction that the mapping
$$(\tau,r) \in \Delta \longmapsto \mathbf{B}(\tau)\U(\tau,r)f_n(r)$$
belongs to $L^1(\Delta,\e_+)$. Therefore,
\begin{equation*}\begin{split}\int_s^t \left\|f_{n+1}(\t)\right\|_\e \d\t &=\int_s^t \d\t\int_s^{\t} \left\|\mathbf{B}(\t)\U(\t,r)f_n(r)\right\|_\e \d r\\
&=\int_s^t \d r \int_{r}^t \left\|\mathbf{B}(\t)\U(\t,r)f_n(r)\right\|_\e \d \t\end{split}\end{equation*}
thanks to Fubini's theorem. From \eqref{estimalpha},
$$\int_{r}^t \left\|\mathbf{B}(\t)\U(\t,r)f_n(r)\right\|_\e \d \t \leq \|f_n(r)\|_\e -\|\U(t,r)f_n(r)\|_\e$$
so that
\begin{equation*}\int_s^t \left\|f_{n+1}(\t)\right\|_\e \d\t \leq \int_s^t \|f_n(r)\|_\e \d r - \int_s^t \|\U(t,r)f_n(r)\|_\e \d r\end{equation*}
which is exactly \eqref{estimateoverlineV} by definition of $f_n(r)$. Notice that, in the \emph{formally conservative} case, the last two inequalities are identities.
\end{proof}
An important consequence of the above Proposition is the following estimate
\begin{cor}\label{cor:cor19} For any $n \in \mathbb{N}$ and any $u \in \e_+$ one has
$$\sum_{k=0}^n \left\|\overline{\V}_k(t,s)u\right\|_\e \leq \|u\|_\e - \int_s^t \|{\mathbf{B}}(\t)\overline{\V}_{n}(\t,s)u\|_\e \d\t$$
with the equality sign in the \emph{formally conservative} case.
\end{cor}
\begin{proof} The proof follows easily from the above Proposition. Indeed, one has
\begin{equation*}\begin{split}
\|\overline{\V}_{n+1}(t,s)u\|_\e &=\left\|\int_s^t \U(t,r){\mathbf{B}}(r)\overline{\V}_n(r,s)u \d r\right\|_\e\\
&=\int_s^t \left\|\U(t,r){\mathbf{B}}(r)\overline{\V}_n(r,s)u\right\|_\e\d r\\
&\leq \int_s^t \|{\mathbf{B}}(\t)\overline{\V}_{n}(\t,s)u\|_\e \d \t - \int_s^t\|{\mathbf{B}}(\t)\overline{\V}_{n+1}(\t,s)u\|_\e \d\t
\end{split}\end{equation*}
where the last inequality is exactly \eqref{estimateoverlineV}. Notice that the last inequality becomes an identity in the \emph{formally conservative} case. Then, using \eqref{estimalpha}, one obtains easily the result by induction.
\end{proof}

With the above results in hands, one proves the following along the same lines of Proposition \ref{luisanotes2} (details are left to the reader):
\begin{propo}\label{luisanotes3}
For any $n \geq 0$  one has
\begin{enumerate}[(1)\:]
\item for any  $0 \leq s \leq r \leq t$:
$$\overline{\V}_n(t,s)u=\sum_{k=0}^n\overline{\V}_k(t,r)\overline{\V}_{n-k}(r,s)u,\qquad \forall u \in \e.$$
\item The mapping $(t,s) \in \Delta \mapsto
\overline{\V}_n(t,s)$ is strongly continuous.
\end{enumerate}
\end{propo}
 Our goal is now to prove that $\overline{\V}_n(t,s)=\V_n(t,s)$ for any $n \in \mathbb{N}$ and any $(t,s) \in \Delta.$ To do so, for technical reason, we shall need here some stronger measurability than Assumption \ref{hypo1}. Namely, we shall assume here the following
\begin{hyp}\label{hypo2} Let $\Delta=\{(t,s) \;;\;0 \leq s \leq t < \infty\}$ and $\Delta_2=\{(t,s,r)\;;\;0 \leq r \leq s \leq t < \infty\}.$ In addition to Assumptions \ref{hypo}, we assume that, for any $f \in L^1(\Delta,\e)$, the mapping
$$(t,s,r) \in \Delta_2 \longmapsto \mathbf{B}(t)\U(t,s)f(s,r) \in \e$$
is measurable.
\end{hyp}
\begin{nb} It can be seen that Assumption \ref{hypo2} includes our previous Assumption \ref{hypo1} as a particular case. Moreover, if the space $\e$ is a space of type $\mathrm{L}$ (typically, if $\e$ is a $L^1$-space, see \cite[Theorem 2.39]{arloban}) then the results of the following section would hold under the sole Assumption  \ref{hypo1} (see Remark \ref{nbspaceL} after Proposition \ref{Luisa}).
\end{nb}
We begin with a technical lemma where the above Assumption \ref{hypo2} is fully exploited:
\begin{lemme}\label{lem:mesurabilite} For any $f \in \X$, $n \in \mathbb{N}$ the mapping
$$(t,s) \in \Delta \longmapsto \mathbf{B}(t)\overline{\V}_n(t,s)f(s) \in \e$$
is measurable. In particular, for any $t \geq 0$, the mapping
$$(\t,s) \in (0,t)\times (t,\infty) \longmapsto \U(s,s-\t)\mathbf{B}(s-\t)\overline{\V}_n(s-\t,s-t)f(s-t) \in \e$$
is measurable.
\end{lemme}
\begin{proof} One first notices that, for any $n \geq 1$
$$\mathbf{B}(s_2)\overline{\V}_n(s_2,s_1)f(s_1)=\int_{s_1}^{s_2} \mathbf{B}(s_2)\U(s_2,r)\mathbf{B}(r)\overline{\V}_{n-1}(r,s_1)f(s_1)\d r \qquad \forall s_2 \geq s_1 \geq 0.$$
Then one easily proves by induction that Assumption \ref{hypo2} implies that the mapping
$$(s_1,s_2) \in \Delta \longmapsto \mathbf{B}(s_2)\overline{\V}_n(s_2,s_1)f(s_1) \in \e$$
is measurable for any $n \in\mathbb{N}$. The conclusion follows.
\end{proof}

With this in hands, let us introduce the following definition:
\begin{defi}\label{defi:Tnover}
For any $f \in \X$, $t \geq 0$ and $n \in \mathbb{N}$, define
$$\left[\overline{\T}_n(t)f\right](r)=\chi_+(r-t)\overline{\V}_n(r,r-t)f(r-t) \qquad r \geq 0.$$
\end{defi}
 It is easily seen that $\overline{\T}_n(t)$ is a bounded operator in $\X$ for any $t \geq 0$ and any $n \in \mathbb{N}$.
One proves the following
\begin{propo}\label{Luisa} Assume that Assumptions \ref{hypo2} and \ref{hypB} are in force.
For any $n \in \mathbb{N}$, one has $\T_n(t)=\overline{\T}_n(t)$ for any $t\geq 0$.
\end{propo}
\begin{proof} Let us prove the first part of the result by induction. It is cleat that, for $n=0$, $\overline{\T}_0(t)={\T}_0(t)$ for any $t \geq 0$. Assume now that there is some $n \geq 0$ such that
$$\overline{\T}_n(t)f={\T}_n(t)f \qquad \forall f \in \X, \:t \geq 0.$$
Let $f \in \D(\Z)$ be fixed. According to the definition of $\overline{\T}_{n+1}(t)$ one has for any $s \geq t \geq 0$
\begin{multline}\label{tn+1-1}
\big[\overline{\T}_{n+1}(t)f\big](s)=\overline{\V}_{n+1}(s,s-t)f(s-t)\\
=\int_{s-t}^s \U(s,r)\mathbf{B}(r)\overline{\V}_{n}(r,s-t)f(s-t)\,\d r\\
=\int_0^t \U(s,s-\t)\mathbf{B}(s-\t)\overline{\V}_{n}(s-\t,s-t)f(s-t)\,\d\t.
\end{multline}
Now, for any $s \geq t \geq \t \geq 0$, one has
$$\overline{\V}_{n}(s-\t,s-t)f(s-t)=\big[\overline{\T}_n(t-\t)f\big](s-\t)=\big[\T_n(t-\t)f\big](s-\t)$$
according to Definition \ref{defi:Tnover} and the induction hypothesis. At this stage, one makes the following observation: since  $\T_n(t-\t)f \in \D(\Z)$ for any $t \geq \t \geq 0$ and since any $g \in \D(\Z)$ is such that $g(s) \in \D(\mathbf{B}(s))$ with $[\widehat{\fB}g](s)=\mathbf{B}(s)g(s)$ for any $s \geq 0$ (see Proposition \ref{closable}), one gets $\T_n(t-\t)f \in \D(\widehat{\fB})$ for any $s \geq \t \geq 0$ with, for any fixed $\t \in [0,t]$ one has
\begin{equation}\label{Bs1t1}
\mathbf{B}(s-\t)\overline{\V}_n(s-\t,s-t)f(s-t)=\left[\widehat{\fB}\T_n(t-\t)f\right](s-\t) \qquad \text{ for a. e. } s \geq t.\end{equation}
Consequently,
$$\U(s,s-\t)\mathbf{B}(s-\t)\overline{\V}_n(s-\t,s-t)f(s-t)=\left[\T_0(\t)\widehat{\fB}\T_n(t-\t)f\right](s) \qquad \text{ for a. e. } s \geq t.$$
 Then, using Lemma \ref{lem:mesurabilite} together with \cite[Lemma 3.2]{miya}, one has
\begin{multline}\label{integLL}
\int_0^t \U(s,s-\t)\mathbf{B}(s-\t)\overline{\V}_n(s-\t,s-t)f(s-t)\d \t=\left[\int_0^t \T_0(\t)\widehat{\fB}\T_n(t-\t)f\d \t\right](s) \\ \text{ for almost any } s \geq t.
\end{multline}
Combining \eqref{tn+1-1} and \eqref{integLL} we get therefore
$$\overline{\T}_{n+1}(t)f =\int_0^t  \T_0(\t)\widehat{\fB}\T_n(t-\t)f\d \t  .$$
Now, according to Lemma \ref{lem:Tn+1bis}, one obtains that
$$\overline{\T}_{n+1}(t)f=\T_{n+1}(t)f  \qquad \forall  t \geq 0.$$
This identity is true for any $f \in \D(\Z)$. Since $\D(\Z)$ is dense in $\X$ and both $\T_n(t)$ and $\overline{\T}_n(t)$ are bounded operators, we get the first part of the Proposition.  With Definition \ref{defi:Tnover}, the fact that $\T_n(t)=\overline{\T}_n(t)$ translates immediately into $\overline{\V}_n(t,s)=\V_n(t,s)$ for any $(t,s) \in \Delta$. \end{proof}
\begin{nb}\label{nbspaceL} Notice that, if $\e$ is an $\mathrm{L}$-type space (see \cite[p. 39]{arloban}) then \eqref{integLL} will hold under the sole Assumption \ref{hypo1}.
\end{nb}

As a general consequence of the above results, we have the following
\begin{propo}\label{decroit} Assume that Assumptions \ref{hypo2} and \ref{hypB} are in force. For any $u \in \e_+$ and any $n \in \mathbb{N}$,
$$ {\V}_n(t,s)u \in \D({\mathbf{B}}(t)) \quad \text{ for almost every } t \geq s$$
with
$$\int_s^t \|\mathbf{B}(\t)\V_{n+1}(\t,s)u\|_\e \d\t \leq \int_s^t \|\mathbf{B}(\t)\V_n(\t,s)u\|_\e \d\t \leq \|u\|_\e \qquad \forall n \in \mathbb{N}.$$ Moreover,
$$\sum_{k=0}^n \left\| {\V}_k(t,s)u\right\|_\e \leq \|u\|_\e - \int_s^t \|{\mathbf{B}}(\t) {\V}_{n}(\t,s)u\|_\e \d\t$$
with the equality sign in the \emph{formally conservative} case.
\end{propo}
\begin{proof} Let $n \in \mathbb{N}$ be fixed. Since $\overline{\T}_n(t)=\T_n(t)$ for any $t\geq 0$ , one also has $\overline{\V}_n(t,s)=\V_n(t,s)$ for any $(t,s) \in \Delta$. Then, for any $u \in \e_+ $, $s \geq 0$, $\V_n(t,s)u \in  \D({\mathbf{B}}(t))$ for almost every $t \geq s  $ according to Proposition \ref{prop:Propimp} and \eqref{estimateoverlineV} reads
$$\int_s^t \|{\mathbf{B}}(\t) {\V}_{n+1}(\t,s)u\|_\e \d \t \leq
\int_s^t \|{\mathbf{B}}(\t) {\V}_{n}(\t,s)u\|_\e \d \t - \int_s^t\|\U(t,\t){\mathbf{B}}(\t) {\V}_{n}(\t,s)u\|_\e \d \t$$
with an equality sign in the \emph{formally conservative} case. The first part of the Proposition follows then easily. The second part follows from a direct application of Corollary \ref{cor:cor19}.
\end{proof}

Moreover, one has the following:
\begin{propo}\label{integrale} Assume that Assumptions \ref{hypo2} and \ref{hypB} are in force. For any $n \in \mathbb{N}$ and any $\t >0$, one has
\begin{equation}\label{fint+Z}\int_0^\t \|\widehat{\fB} \T_n(t)f\|_\X \d t =\int_0^\infty \d r \int_0^\t  \|{\mathbf{B}}(r+t)\V_n(r+t,r)f(r)\|_\e \d t, \quad \forall f \in \D(\Z).\end{equation}
Moreover, for any $f \in \X_+$ and any $n \in \mathbb{N}$
\begin{equation}\label{fint+Zb}\left\|\widehat{\fB} \int_0^\t  \T_n(t)f \d t\right\|_\X =\int_0^\infty \d r \int_0^\t  \|{\mathbf{B}}(r+t)\V_n(r+t,r)f(r)\|_\e \d t \qquad \forall \t >0.\end{equation}
\end{propo}
\begin{proof} Let $f \in \D(\Z)$ (not necessarily in $\X_+$) and $n \in \mathbb{N}$ be fixed. One deduces directly from \eqref{Bs1t1} that
\begin{equation}\label{Bs1t1bis}\left[\widehat{\fB}\T_n(t)f\right](s)=\mathbf{B}(s)\V_n(s,s-t)f(s-t) \qquad \text{ for a. e. } s \geq t \geq 0.\end{equation}
Using Point \textit{(2)} of Remark \ref{importantNB} and since $\widehat{\fB}\T_n(t)f=\T_{n+1}(t)\Z f +\T_{n}(t)\widehat{\fB} f-\Z\T_n(t)f$, the mapping $t > 0 \mapsto \widehat{\fB}\T_n(t)f \in \X$ is continuous. In particular, since $\X=L^1((0,\infty),\e)$, one sees that this implies that  for any $\t >0$,  $$t \in (0,\t) \longmapsto \int_0^\infty \| {\mathbf{B}}(s)\V_n(s,s-t)f(s-t)\chi_+(s-t)\|_\e \d s<\infty$$ is a continuous mapping. It is in particular integrable so that
$$\int_0^\t \d t \int_t^\infty \left\|{\mathbf{B}}(r)\V_n(r,r-t)f(r-t)\right\|_\X  \d r <\infty \qquad \forall \t >0. $$
Then, the identity follows directly from \eqref{Bs1t1bis} since the mapping $(s,t) \mapsto \mathbf{B}(s)\V_n(s,s-t)f(s-t) \in \e$ is measurable (see Lemma \ref{lem:mesurabilite}). Thus, \eqref{fint+Z} holds for any $f \in \D(\Z)$. Moreover, if $f \in  \X_+$ one also has that
$$\left\|\widehat{\fB}\int_0^\t \T_n(t)f\d t\right\|_\X \leq \|f\|_\X \qquad \forall \t >0$$
while, for almost every $r >0$ it holds
$$\int_0^\t  \|{\mathbf{B}}(r+t)\V_n(r+t,r)f(r)\|_\e \d t=\int_r^{r+\t} \|\mathbf{B}(s)\V_n(s,r)f(r)\|_\e \d s \leq \|f(r)\|_\e \qquad \forall \t >0.$$
As a consequence, for any $f \in \X$, one has, for almost every $r >0$,
\begin{multline}\label{eqgam1}\left\|\widehat{\fB}\int_0^\t \T_n(t)f\d t\right\|_\X  \leq 2\gamma \|f\|_\X \qquad \text{ and } \\
\int_r^{r+\t} \|\mathbf{B}(s)\V_n(s,r)f(r)\|_\e \d s \leq 2\gamma \|f(r)\|_\e \qquad \forall \t >0.\end{multline}
Now, if $f \in \X_+$ is given, there exists a  sequence $(f_k)_k \subset \D(\Z) \cap \X_+$ such that $\|f_k - f\|_\X \to 0$ as $k \to \infty.$ Up to a subsequence, $\lim_{k \to \infty}\|f_k(r)-f(r)\|_\e=0$ for almost every $r >0.$  According to \eqref{eqgam1}, one deduces from the dominated convergence theorem that
$$\lim_{k \to \infty} \|\widehat{\fB} \int_0^\t\T_n(t)f_k\d t\|_\X = \|\widehat{\fB}\int_0^\t \T_n(t) f \d t\|_\X$$
and
$$\lim_{k \to \infty}\int_0^\infty \d r\int_0^\t  \|{\mathbf{B}}(r+t)\V_n(r+t,r)f_k(r)\|_\e \d t=\int_0^\infty \d r\int_0^\t  \|{\mathbf{B}}(r+t)\V_n(r+t,r)f(r)\|_\e \d t $$
which gives the conclusion.\end{proof}


We end this section with a technical Lemma, related to the above Proposition:
\begin{lemme}\label{techn:lem} Assume that Assumptions \ref{hypo2} and \ref{hypB} are in force.
 For any $u \in \e_+$ and any nonnegative $\varphi
\in \mathscr{C}_c(\mathbb{R}_+)$, let
$$f(\cdot)=\varphi(\cdot)u \in \X_+.$$ Then,
\begin{equation}\label{limSE}
\lim_{n \to \infty} \left\|\widehat{\fB}\int_0^t \T_n(r)f\d r\right\|_{\X}=\int_0^\infty \varphi(r)\left(\lim_{n \to \infty} \int_{r}^{r+t}\left\|\mathbf{B}(\t)\V_n(\t,r)u \right\|_\e\d \t\right)\d r
\end{equation}
holds for almost every $t \geq 0.$
\end{lemme}
\begin{proof} Since $f \in \X_+$, Eq. \eqref{fint+Zb} asserts that, for any $n \in \mathbb{N}$
$$\| \widehat{\fB} \int_0^t\T_n(s)f\d s\|_\X  =\int_0^\infty \varphi(r)\left(\int_{r}^{r+t}\left\|\mathbf{B}(\t)\V_n(\t,r)u \right\|_\e\d \t\right)\d r, \qquad \forall  t >0.$$
Since $\int_{r}^{r+t}\left\|\mathbf{B}(\t)\V_n(\t,r)u \right\|_\e \leq \|u\|_\e$ for any $n\geq 1$ and almost every $r, t >0$, we get the conclusion thanks to the dominated convergence theorem.
\end{proof}

\subsection{Honesty of $\vt$}\label{sec:hovt} We have now all in hands to \emph{characterize}, in terms of $\mathbf{B}(t)$ and $\V_n(t,s)$ only,  the honesty of the family $\vt$. Precisely, we prove the following

\begin{theo}\label{theo:main?} The evolution family $\vt$ is honest if and only if
\begin{equation}\label{eq:defihonest}
\lim_{n \to \infty} \int_s^t \left\|\mathbf{B}(\t)\V_n(\t,s)u\right\|_\e \d\t=0 \qquad \forall t \geq s\end{equation}
for any $u \in \e_+$.
\end{theo}
\begin{proof}  Assume first that $\vt$ is honest in $\e$, i.e. the semigroup $\Tt$ is honest in $\X$. Let $u \in \e_+$ be fixed.
Notice that the sequence $\left(\int_{r}^{r+t}\left\|\mathbf{B}(\t)\V_n(\t,r)u \right\|_\e\d \t\right)_n$ is nonincreasing for almost any $r \geq 0$ and that, given $t >0$, the mapping
$$r \in \mathbb{R}^+ \longmapsto \lim_{n \to \infty} \int_{r}^{r+t}\left\|\mathbf{B}(\t)\V_n(\t,r)u \right\|_\e\d \t$$
is continuous.  According to \eqref{limSE}, for any $\varphi \in \mathscr{C}_c(\mathbb{R}_+)$ one has
$$\lim_{n \to \infty} \int_{r}^{r+t}\left\|\mathbf{B}(\t)\V_n(\t,r)u \right\|_\e\d \t =0 \qquad \forall r \in \mathrm{Supp}(\varphi), \: \forall t\geq 0.$$
Since $\varphi$ is arbitrary, one sees that
$$\lim_{n \to \infty} \int_s^t \|\mathbf{B}(\t)\V_n(\t,s)u\|_\e \d\t=0 \qquad \forall t \geq s$$
which gives the first implication.

Conversely, assume that \eqref{eq:defihonest} holds true for any $u \in \e_+$ and let $f \in \X_+$ and $\t >0$ be fixed. Define
$$F_n(r)=\int_0^\t  \|{\mathbf{B}}(r+t)\V_n(r+t,r)f(r)\|_\e \d t, \qquad \forall r >0.$$
Since $f(r) \in \e_+$ for a. e. $r > 0$, one has
\begin{equation}\label{almo}\lim_{n \to \infty} F_n(r)=0 \qquad \text{ for a. e. } r >0.\end{equation}
Now, according to Prop. \ref{decroit}, for any $r>0,$ the sequence $(F_n(r))_n$ is nonincreasing with
$$F_n(r) \leq \|f(r)\|_\e \qquad \forall r >0, \:n \in \mathbb{N}$$
while, Proposition \ref{integrale} asserts that
$$\int_0^\infty F_n(r)\d r=\left\|\widehat{\fB}\int_0^\t \T_n(t)f\d t\right\|_\X $$
It is clear now that \eqref{almo} together with Lebesgue Theorem implies that
$$\lim_{n \to \infty}\left\|\widehat{\fB}\int_0^\t \T_n(t)f\d t\right\|_\X =0.$$
Since $f \in \X_+$ is arbitrary, this exactly means that the whole semigroup is honest.
\end{proof}

\begin{nb} It is possible now to extend the definition of honesty as follows: given $u \in \e_+$ and $s \geq 0$, the trajectory $\left(\V(t,s)u\right)_{t \geq s}$ is said to be honest if \eqref{eq:defihonest} holds true.
\end{nb}

We conclude this section with some considerations about honesty theory in the formally conservative case. Recall that, by definition, in such a case, the semigroup is honest if and only if it is conservative. We can observe that such a property still holds for the evolution family. Indeed, according to Proposition \ref{decroit}, in the formally conservative case we have, for any $u \in\e_+$ and any $t \geq s \geq 0$
\begin{equation*}
\|\V(t,s)u\|=\lim_{n \to \infty}\sum_{k=0}^n \|\V_k(t,s)u\|=\|u\|-\lim_{n \to \infty}\int_s^t \|\mathbf{B}(\t)\V_n(\t,s)u\| \d \t.
\end{equation*}
Therefore, the family $\vt$ is honest if and only if
$$\|\V(t,s)u\|=\|u\| \qquad \forall u \in \e_+, \;\forall t \geq s\geq 0$$
which exactly means that the evolution family $\vt$ is honest if and only if it is conservative.

\section{Application to the linear Boltzmann equation}

We apply the results of the previous part of this paper to a large class of non-autonomous linear kinetic equations.
\subsection{Spatially homogeneous linear kinetic equation}

We consider the non-autonomous (spatially homogeneous) kinetic equation in $L^1(V,\d\mu(v))$:
\begin{equation}\label{spatiallyhomogeneous}
\left\{\begin{split}
\partial_t f(v,t)+\Sigma(v,t) f(v,t)&=\int_{V} b(t,v,v')f(v',t)\d\mu(v') \qquad v  \in V\\
f(v,s)&=h(v) \in L^1(V,\d\mu(v))
\end{split}\right.
\end{equation}
for some $s \geq 0.$ In the above, $\mu$ is a Borel measure over $\mathbb{R}^d$ $(d\geq 1)$ with support $V$ while the  collision frequency $\Sigma$ and the collision kernel $b(t,\cdot,\cdot)$ satisfy the following assumptions
\begin{hyp}\label{hypSiBmu} \begin{enumerate}

  \item The mapping $(t,v) \in \R^+ \times V \mapsto \Sigma(t,v)$ is measurable and non-negative. Moreover, for $\mu$-almost every $v \in V$, the mapping $t \mapsto \Sigma(t,v)$ is locally integrable on $\R^+$.
\item The collision kernel $b$ is such that $(t,v,v') \in \R^+ \times V \times V \mapsto b(t,v,v')$ is measurable and non-negative. Moreover,
\begin{equation}\label{subcritic}
\int_V b(t,v,v') \d\mu(v) \leq \Sigma(t,v') \qquad \text{ for almost every } (t,v') \in \R^+ \times V.\end{equation}
\end{enumerate}
\end{hyp}
\begin{nb} The above inequality \eqref{subcritic} is a classical assumption in neutron transport theory and is usually referred to as the \emph{sub-critical hypothesis}.
\end{nb}

One will see \eqref{spatiallyhomogeneous} as a perturbation of the simpler equation
$$\partial_t f(v,t)+\Sigma(v,t)f(v,t)=0 \qquad f(v,0)=f_0(v)$$
and the associated evolution family
\begin{equation}\label{U0ts}
\U_h(t,s)\varphi(v)=\exp\left(-\int_s^t \Sigma(v,\tau)\d\tau\right)\varphi(v), \qquad t\geq s\geq 0, \qquad \varphi \in \e_h:=L^1(V,\d\mu).\end{equation}
We begin with the following smoothing effect of $\U_h(t,s)$:
\begin{lemme}\label{smooth}
For any non-negative $\varphi \in \e_h$ and any $s \geq 0$  one has $\U_h(t,s)\varphi \in L^1(V,\Sigma(t,\cdot)\,\d\mu)$ for almost every $t \in [s,\infty)$,  i.e.
\begin{equation}\label{trajectorial}\int_V \U_h(t,s)\varphi(v)\,\Sigma(t,v)\d\mu(v) < \infty \qquad \text{ for a.e. } t \in [s,\infty).\end{equation}
\end{lemme}

\begin{proof} Let $s >0$ and $\varphi \in L^1(V,\d\mu)$, $\varphi \geq 0$ $\mu$-a.e. be fixed. One defines the mapping $\Q\::\:t \in [s,\infty) \mapsto \Q(t) \in [0,\infty]$ by
$$\Q(t)=\int_V \Sigma(t,v)\U_h(t,s)\varphi(v)\d\mu(v)=\int_V \Sigma(t,v)\exp\left(-\int_s^t \Sigma(\t,v)\d\t\right)\varphi(v)\d\mu(v).$$
For a given $T >s$, one has, thanks to Fubini's Theorem
$$\int_s^T \Q(t)\d t=\int_V \varphi(v)\left(\int_s^T \Sigma(t,v)\exp\left(-\int_s^t \Sigma(\t,v)\d\t\right)\d t\right)\d\mu(v).$$
Now, for any fixed $v \in V$, since the mapping $t\mapsto \Sigma(t,v)$ belongs to $L^1_{\mathrm{loc}}(\R^+)$, the set $\mathcal{L}_{\Sigma(\cdot,v)} \subset \R^+$ of its Lebesgue points is such that its complementary $\mathcal{L}_{\Sigma(\cdot,v)}^c$ is of zero Lebesgue measure and
$$\Sigma(t,v)\exp\left(-\int_s^t \Sigma(\t,v)\d\t\right)=-\dfrac{\d}{\d t}\exp\left(-\int_s^t \Sigma(\t,v)\d\t\right) \qquad \forall  t \in \mathcal{L}_{\Sigma(\cdot,v)}\cap (s,T).$$
This allows to compute the integral $\int_s^T \Sigma(t,v)\exp\left(-\int_s^t \Sigma(\t,v)\d\t\right)\d t $ for any $v \in V$ and yields
\begin{equation}\label{intsTQ}\int_s^T \Q(t)\d t=\int_V \varphi(v)\left(1-\exp\left(-\int_s^T \Sigma(\t,v)\d\t\right)\right)\d\mu(v) < \infty \quad \forall T > s.\end{equation}
This proves in particular that $\Q(t) < \infty$ for almost every $t \in [s,\infty).$ This achieves the proof.
\end{proof}
 \begin{nb} Notice that the above smoothing effect of $\U_h(\cdot,\cdot)$ is actually a "trajectorial smoothing effect" in the following sense: for any $\varphi \in \e_h$ and any $s >0$, there exists a set $\Xi_{s}^\varphi  \subset [s,\infty)$ such that $[s,\infty) \setminus \Xi_{s}^\varphi$ is of zero Lebesgue measure and \eqref{trajectorial} holds for any $t \in \Xi_{s}^\varphi$. Clearly, the set $\Xi_{s}^\varphi$ depends on the choice of $\varphi.$ In this sense, the smoothing effect is related to the trajectory $t \in (s,\infty)  \mapsto \U_h(t,s)\varphi$ under investigation.
\end{nb}

\begin{nb} Notice  that, for any $\varphi \in L^1(V,\d\mu)$ and any $t \geq s \geq 0$, with the notations of the above proof,
\begin{equation*}\begin{split}
\int_{s}^{t}\Q(r)\d r &=\int_{V}\left(1-\exp(-\int_s^T \Sigma(\t,v)\d\t\right)\varphi(v)\d\mu(v)\\
&=\int_{V}\varphi (v)\d\mu (v)-\int_{V}\exp\left(-\int_s^{t}\Sigma (\tau ,v)d\tau\right)\varphi (v)d\mu (v)
\end{split}\end{equation*}
which implies that the mapping
\[
\left[ s,+\infty \right[ \ni t \longmapsto \int_{V}\exp\left(-\int_s^{t}\Sigma (\tau ,v)d\tau\right)\varphi (v)d\mu (v)
\]%
is differentiable at any $t \in \mathcal{L}_{\Q}$ (where, $\mathcal{L}_\Q$ is the set of Lebesgue point of $\Q$) with derivative given by
$-\int_V \Sigma(t,v)\exp\left(-\int_s^{t}\Sigma (\tau ,v)d\tau\right)\varphi (v)d\mu (v).$
\end{nb}
We define then the family of perturbations $(\mathbf{B}_h(t))_{t \geq 0}$ as
\begin{equation}\label{B0t}
\mathbf{B}_h(t)\varphi(v)=\int_{V}b(t,v,v')\varphi(v')\d\mu(v')\end{equation}
with domain
$$\D(\mathbf{B}_h(t))=\left\{\varphi \in L^1(V,\d\mu)\,,\,\int_V b(t,\cdot,v')|\varphi(v')|\d\mu(v') \in L^1(V,\d\mu)\right\} \qquad \forall t \geq 0.$$
Defining
$$\sigma(t,v)=\int_{V} b(t,v',v)\d\mu(v') \qquad \forall v \in V, \quad t \geq 0$$
one sees easily that
$$\D(\mathbf{B}_h(t))=\left\{\varphi \in L^1(V,\d\mu)\,,\,\int_V \sigma(t,v)|\varphi(v)|\d\mu(v) < \infty\right\}=L^1(V,\sigma(t,v)\d\mu(v)) \qquad \forall t \geq 0.$$
The above smoothing effect has an interesting consequence:
\begin{propo}\label{propo:homo} Let Assumptions \ref{hypSiBmu} be in force. Then, for any $\varphi \in  L^1(V,\d\mu)$ and any $s \geq 0$ one has $\U_h(t,s)\varphi \in \D(\mathbf{B}_h(t))$ for almost-every $t \in [s,\infty)$; the mapping $t \in [s,\infty) \mapsto \mathbf{B}_h(t)\U_h(t,s)\varphi \in L^1(V,\d\mu)$ is measurable with
$$\int_s^t \|\mathbf{B}_h(\t)\U_h(\t,s)\varphi\|_{L^1(V,\d\mu)} \d\t \leq \|\varphi\|_{L^1(V,\d\mu)} - \|\U_h(t,s)\varphi\|_{L^1(V,\d\mu)} \qquad \forall t \geq s$$
with the equality sign if
\begin{equation}\label{sigmaSigma}
\sigma(t,v)=\Sigma(t,v) \qquad \text{ for almost every } (t,v) \in \mathbb{R}^+ \times V.\end{equation}\end{propo}
\begin{proof}  Let $\varphi \in  L^1(V,\d\mu)$ and  $s \geq 0$ be fixed. The mapping
$$(t,v,v') \in [s,\infty) \times V \times V \longmapsto b(t,v,v')\exp\left(-\int_s^t \Sigma(\t,v')\d \t\right)\,\varphi(v') \in \mathbb{R}^+$$
is measurable and is actually integrable. Thanks to Fubini's theorem and using the notations of the previous Proposition, one has for any $T > s$
\begin{multline*}
\int_s^T \d t \int_V \d\mu(v)\int_V   b(t,v,v')\exp\left(-\int_s^t \Sigma(\t,v')\d \t\right)\,\varphi(v')\d\mu(v')\\
=\int_s^T \d t \int_V \sigma(t,v')\exp\left(-\int_s^t \Sigma(\t,v')\d \t\right)\,\varphi(v')\d\mu(v')\\
\leq \int_s^T \d t \int_V \Sigma(t,v')\exp\left(-\int_s^t \Sigma(\t,v')\d \t\right)\,\varphi(v')\d\mu(v')=\int_s^T \Q(t)\d t
\end{multline*}
with the equality sign if $\sigma(t,v)=\Sigma(t,v)$ for almost every $t,v$.  According to \eqref{intsTQ}, one sees that
$$\int_s^T \Q(t)\d t=\|\varphi\|_{L^1(V,\d\mu)}-\|\U_h(T,s)\varphi\|_{L^1(V,\d\mu)} \qquad \forall T > s$$
i.e.
\begin{multline*}
\int_s^T \d t \int_V \d\mu(v)\int_V   b(t,v,v')\exp\left(-\int_s^t \Sigma(\t,v')\d \t\right)\,\varphi(v')\d\mu(v')\\
\leq \|\varphi\|_{L^1(V,\d\mu)}-\|\U_h(T,s)\varphi\|_{L^1(V,\d\mu)} \qquad \forall T > s \end{multline*}
with an identity if \eqref{sigmaSigma} holds. Letting then $T \to \infty$ we get
$$\int_s^\infty \d t \int_V \d\mu(v)\int_V   b(t,v,v')\exp\left(-\int_s^t \Sigma(\t,v')\d \t\right)\,\varphi(v')\d\mu(v')\leq \|\varphi\|.$$
Considering that, for almost every $t \geq s$ and almost every $v \in V$ it holds
$$\int_V   b(t,v,v')\exp\left(-\int_s^t \Sigma(\t,v')\d \t\right)\,\varphi(v')\d\mu(v')=\left[\mathbf{B}_h(t)\U_h(t,s)\varphi\right](v)$$
we can invoke Fubini's theorem to conclude that $\mathbf{B}_h(t)\U_h(t,s)\varphi \in \D(\mathbf{B}_h(t))$ for a. e. $t \geq s.$ Moreover, the mapping
$t \in [s,\infty) \mapsto \mathbf{B}_h\U_h(t,s)\varphi \in L^1(V,\d\mu)$ is integrable and
$$\int_s^T \|\mathbf{B}_h\U_h(t,s)\varphi\|_{L^1(V,\d\mu)}\d t \leq \|\varphi\|_{L^1(V,\d\mu)} -\|\U_h(T,s)\varphi\|_{L^1(V,\d\mu)} \quad \forall T > s$$
which is an identity if \eqref{sigmaSigma} holds.
\end{proof}
 The above Proposition shows that, in the space $\e_h=L^1(V,\d\mu)$, the family of perturbation operators $(\mathbf{B}_h(t),\D(\mathbf{B}_h(t)))_{t\geq 0}$ and the unperturbed evolution family $\{\U_h(t,s)\}_{t\geq s}$ satisfy Assumptions \ref{hypo}. It is easy to show that, under Assumptions \ref{hypSiBmu}, the family of perturbation operators $(\mathbf{B}_h(t),\D(\mathbf{B}_h(t)))_{t\geq 0}$ and the unperturbed evolution family $\{\U_h(t,s)\}_{t\geq s}$ also satisfy Assumptions \ref{hypo1} and \ref{hypo2}. Indeed, let us define the space
$$\X=L^1(\mathbb{R}_+,\e_h)=L^1(\mathbb{R}_+ \times V,\d t \d\mu(v)$$
and let $f \in \X$  be given. The mapping
$$(t,v,s,v') \in \left(\mathbb{R}_+ \times V\right)^2 \mapsto b(t,v,v')\exp\left(-\int_s^t \Sigma(\t,v')\d \t\right)f(s,v')\chi_+(t-s)$$
is measurable. Moreover, it is integrable and, thanks to Fubini's theorem, with considerations similar to those used in the previous proof, we can state that $\U_h(t,s)f(s) \in \D(\mathbf{B}_h(t))$ for almost every $(t,s) \in \Delta$ and that the mapping $(t,s) \in \Delta \mapsto \mathbf{B}_h(t)\U_h(t,s)f(s) \in \e_h$ is integrable. This shows that Assumptions \ref{hypo1} are satisfied. Similar considerations show that Assumptions \ref{hypo2} are also met.

Thanks to these considerations, the general abstract approach developed in Sections 2, 3.1 and 3.2 can be applied here. In particular, Theorem \ref{generevol} applies and one can define an evolution semigroup $\left(\V_h(t,s)\right)_{t\geq s}$ in $\e_h$ given by
\begin{equation}\label{V0t}
\V_h(t,s)\varphi =\U_h(t,s)\varphi + \int_s^t \V_h(t,r)\mathbf{B}_h(r)\U_h(r,s)\varphi \d r,
\qquad \forall \varphi \in L^1(V,\d\mu).
\end{equation}
The $C_0$-semigroup $(\T_{0,h}(t))_{t\geq 0}$ in $\X$ associated to $\left(\U_h(t,s)\right)_{t \geq s}$ is defined by \eqref{lien} and we denote by  $(\Z_h,\D(\Z_h))$ its generator while one associates to the evolution family $\left(\V_h(t,s)\right)_{t\geq s\geq 0}$ a perturbed evolution semigroup $(\T_h(t))_{t \geq 0}$ with generator $(\G_h,\D(\G_h))$. To the perturbation family $(\mathbf{B}_h(t))_{t \geq 0}$ one associates an operator $\fB$ in $\X$ as in \eqref{deffB} and an extension $\widehat{\fB}$ of $\fB$ is provided by Theorem \ref{construcB} so that $\G_h$ is an extension of $(\Z_h+\widehat{\fB},\D(\Z_h))$ in $\X$. Let us prove now that Assumption \ref{hypo2} is met. Precisely,
\begin{propo}\label{propFbintegral} The perturbed family $(\mathbf{B}_h(t))_{t \geq 0}$ and the evolution family $(\U_h(t,s))_{t \geq s}$ satisfy Assumption \ref{hypo2}. As a consequence, $\widehat{\fB}=\fB\vert_{\D(\Z)}$ and, for any $g \in \D(\Z_h)\subset \X$, one has
$$\int_0^\infty \d t \int_V \sigma(t,v)|g(t,v)|\d\mu(v) < \infty$$
and
$$[\widehat{\fB}g](t,v)=[\mathbf{B}_h(t)g(t,\cdot)](v)=\int_V b(t,v,v')g(t,v')\d\mu(v') \qquad t \geq 0,v \in V.$$
\end{propo}
\begin{proof} For a fixed $t \geq 0$, let $(\varphi_n)_n$ be a nonnegative increasing sequence in $L^1(V,\d\mu)$ with $\varphi_n \in \D(\mathbf{B}_h(t))$ and $\sup_n \|\mathbf{B}_h(t)\varphi_n\| < \infty$ and $\lim_n \varphi_n=\varphi.$ This exactly means that
$$\sup_n \int_V \sigma(v,t)\varphi_n(v)\d\mu(v) < \infty.$$
According to Fatou's Lemma, one has then $\int_V \sigma(v,t)\varphi(v)\d\mu(v) < \infty$ i.e. $\varphi \in \D(\mathbf{B}_h(t))$. Moreover, by the monotone convergence Theorem, $\lim_n \mathbf{B}_h(t)\varphi_n=\mathbf{B}_h(t)\varphi$ in $L^1(V,\d\mu)$. Assumption \ref{hypo2} is then satisfied and Proposition \ref{closable} yields the result.
\end{proof}

  One can actually characterize $\D(\Z_h)$  along the same lines as \cite[Proposition 2.1]{arlotti}:
\begin{propo} Let $W^{1,1}_0(\mathbb{R}_+;V)$ denotes the space of functions $g=g(t,v) \in \X$ such that, for $\mu$-almost every $v \in V$, the mapping
$$t \geq 0 \longmapsto g(t,v) \text{ is absolutely continuous with }  g(0,v)=0$$
together with $\dfrac{\partial g}{\partial t}  \in \X.$ Then, the generator $(\Z_h,\D(\Z_h))$ of $(\T_h(t))_{t\geq 0}$ is given by
$$\Z_h g(t,v)=-\frac{\partial g}{\partial t}(t,v)-\Sigma(v,t)g(t,v),\qquad g \in \D(\Z)$$
with $\D(\Z_h)=\left\{g=g(t,v) \in W^{1,1}_0(\mathbb{R}_+;V)\,\text{ such that } \Sigma \,g \in \X\right\}.$
\end{propo}

We have now everything in hands to apply the substochastic theory developed in the first part of the paper. In this particular example in which the generator $(\Z_h,\D(\Z_h))$ is explicit in the lifted space $\X$, it is more convenient to treat the  problem of the honesty of the evolution family $\left(\V_h(t,s)\right)_{t\geq s\geq 0}$ directly on the perturbed evolution semigroup $(\T_h(t))_{t \geq 0}$. For such a semigroup, known necessary and sufficient conditions for the honesty have been established by the third author in \cite{musAfrika} and one has
\begin{theo}\label{detaim} Assume that there exists some $M=M(t,v) \in \D(\Z_h)$ with $M(t,v) > 0$ for almost every $(t,v) \in \mathbb{R}^+ \times V$ and such that the detailed balance condition
$$b(t,v,v')M(t,v')=b(t,v',v)M(t,v)$$
holds for any $t\geq 0$ and $\mu$-almost every $v,v' \in V$. If there is some $\lambda >0$ such that
$$\lambda M(t,v)+\frac{\partial M}{\partial t}(t,v) \geq 0 \qquad \forall t \geq 0 \quad \text{ and $\mu$-a.e. } v \in V$$
then the evolution family $\left(\V_h(t,s)\right)_{t\geq s\geq 0}$ is honest in $\e.$
\end{theo}
\begin{proof} If $M  \in \D(\Z_h)$ with the above properties exists, then it is satisfies
$$[\widehat{\fB}M](t,v)=\int_V b(t,v,v')M(t,v')\d\mu(v')=\int_V b(t,v',v)M(t,v)\d\mu(v')=\sigma(t,v)M(t,v)$$
and
$$\left[(\Z_h+\widehat{\fB})M\right](t,v)=-\dfrac{\partial M}{\partial t}(t,v) + \left(\sigma(t,v)-\Sigma(t,v)\right)M(t,v).$$
Since $M$ is non-negative and $b(t,\cdot,\cdot)$ is subcritical, $\left(\sigma(t,v)-\Sigma(t,v)\right)M(t,v)\leq 0$ for any $t >0$ and almost any $v \in V$ so that
$$\left[(\Z_h+\widehat{\fB})M\right](t,v) \leq -\dfrac{\partial M}{\partial t}(t,v) \leq \lambda M(t,v) \qquad \forall t > 0, \text{ for a.e. } v \in V.$$
This shows that $M$ is a non-negative sub-eigenfunction of $\Z_h+\widehat{\fB}$, i.e. we found $M \in \D(\Z_h) \cap \X_+$ and $\l >0$ such that
$$\left(\Z_h + \widehat{\fB}\right) M \leq \l\,M.$$
According to \cite[Corollary 3.27]{ALM} (see also \cite{mkvo}), this shows that the trajectory $\left(\T_h(t)M\right)_{t \geq 0}$ is honest and, since $M$ is quasi-interior in $\X$ (it is a strictly positive function in a $L^1$-space), \cite[Theorem 3.30]{ALM} shows that the whole semigroup $(\T_h(t))_{t \geq 0}$ is honest from which the result follows.
\end{proof}

\begin{nb} If there exists  $M_0=M_0(v) \in L^1(V,\d\mu)$ with $M_0(v) >0$ for a.e. $v \in V$ and such that the detailed balance
$$b(t,v,v')M_0(v')=b(t,v',v)M_0(v)$$
holds for any $t\geq 0$ and $\mu$-almost every $v,v' \in V$ and
$$\Sigma(t,v)M_0(v) \in L^1([0,\infty),L^1(V,\d\mu(v))$$
then, for any $\lambda >0$, one can construct some
$$M=M(t,v)=\beta(t)M_0(v) \in \D(\Z_h)$$
satisfying the hypothesis of the above Theorem for any nonnegative $\beta(\cdot) \in W^{1,1}(\mathbb{R}_+)$ such that
$$\beta'(t)+\lambda\beta(t) \geq 0 \quad (\forall t \geq 0), \qquad \beta(0)=0, \qquad \beta(t) > 0 \quad \forall t > 0.$$
Such a function $\beta(\cdot) \in W^{1,1}_0(\mathbb{R}_+)$ is easily constructed.
\end{nb}

\subsection{Spatially inhomogeneous linear kinetic equation}\label{sec:inhomo}

We deal now with the spatially inhomogeneous version of the above \eqref{spatiallyhomogeneous}, i.e. we shall consider the following
\begin{equation}\label{spatiallyinhomogeneous}
\left\{\begin{split}
\partial_t f(x,v,t)+v \cdot \nabla_x f(x,v,t) &+ \Sigma(v,t) f(x,v,t)\\
&=\int_{V} b(t,v,v')f(x,v',t)\d\mu(v') \qquad x \in \mathbf{\Omega}, \qquad v \in V\\
f(x,v,t)&=0 \qquad  \text{ if } (x,v) \in \mathbf{\Gamma}_-\\
f(x,v,s)&=h(x,v) \in L^1(\mathbf{\Omega} \times V,\d x\otimes\d\mu(v))
\end{split}\right.
\end{equation}
for some $s \geq 0.$ Here, $\mathbf{\Omega} \subset \R^d$ is a open subset, $\mu$ is a Borel measure over $\R^d$ with support $V$ and one defines
$$\mathbf{\Lambda}=\mathbf{\Omega} \times V$$
and
$$\mathbf{\Gamma}_-:=\{(y,v) \in \partial \mathbf{\Omega} \times V\;\text{ such that } y=x-\tau(x,v)v \text{ for some } x \in \mathbf{\Omega} \text{ with } \tau(x,v) <\infty\}$$
where
$$\tau(x,v)=\inf\{t > 0\::\;x-tv \notin \mathbf{\Omega}\} \qquad \forall (x,v) \in \mathbf{\Lambda}.$$
From a heuristic viewpoint, $\tau(x,v)$ is the time needed by a particle having the position
$x$ and the velocity $-v \in V$ to reach the boundary $\partial \mathbf{\Omega}$. One can prove  that $\tau(\cdot,\cdot)$ is measurable on $\mathbf{\Lambda}$. Moreover $\tau(x,v) = 0$ for any $(x, v) \in \mathbf{\Gamma}_-$. \medskip

Let us then introduce the state space
$$\e=L^1(\mathbf{\Lambda}\,,\,\d x \otimes \d\mu(v))$$
with usual norm and  the  unperturbed evolution family $(\U(t,s))_{t \geq s\geq 0}$ defined through
$$[\U(t,s)\varphi](x,v)=\exp\left(-\int_s^t \Sigma(v,\tau)\d\tau\right)\varphi(x-(t-s)v,v)\chi_{\{t-s < \tau(x,v)\}}(x,v),$$
for any $t\geq s \geq 0$ and any $f \in \e$.  Since $\Sigma(\t,v)\geq 0$, one notices easily that $(\U(t,s))_{t\geq s}$ is an evolution family in $\e$ with $\U(t,s)f \in \e_+$  and $\|\U(t,s)f\| \leq \|f\|$ for any $f \in \e_+$ and any $t \geq s \geq 0.$ This shows that Assumptions \ref{hypo}, \textit{(i)} and \textit{(ii)} are satisfied.\medskip

Moreover, for any $t \in \R_+$, we denote by $\mathbf{\Sigma}(t)$ the multiplication operator
$$\mathbf{\Sigma}(t)\::\:\D(\mathbf{\Sigma}(t)) \subset \e \longrightarrow \e$$ defined by $\D(\mathbf{\Sigma}(t))=L^1\left(\mathbf{\Lambda},\Sigma(t,v)\d x \otimes \d\mu(v)\right)$ and $$[\mathbf{\Sigma}(t)\varphi](x,v)=\Sigma(t,v)\varphi(x,v) \qquad \forall \varphi \in \D(\mathbf{\Sigma}(t)).$$
As in the spatially homogeneous case, the evolution family $(\U(t,s))_{t\geq s\geq 0}$ has some smoothing effect:
\begin{lemme}\label{lem-technint}For any $f \in \e_+$ and any $s \geq 0$, one has $\U(t,s)f \in \D(\mathbf{\Sigma}(t))$ for almost every $t \in [s,\infty)$.
\end{lemme}
The proof of this Lemma, of technical nature but non trivial, is postponed to Appendix \ref{app:integration}. \medskip

One defines $\mathbf{B}(t)$ as
$$[\mathbf{B}(t)\varphi](x,v)=\int_V b(t,v,v')\varphi(x,v')\d\mu(v') \qquad \forall \varphi \in \D(\mathbf{B}(t))$$
where, as above,
$$\D(\mathbf{B}(t))=\left\{\varphi \in \e\text{ such that } \int_{\mathbf{\Omega} \times V}\sigma(t,v) |\varphi(x,v)|\d x \d\mu(v) < \infty\right\}$$
and
$$\sigma(t,v)=\int_V b(t,v',v)\d\mu(v').$$
We shall always assume that the Assumptions \ref{hypSiBmu} are in force. In particular, one notices that
\begin{equation}\label{DSigDBt}\D(\mathbf{B}(t)) \subset \D(\mathbf{\Sigma}(t)) \qquad \text{ and } \qquad \|\mathbf{B}(t)\varphi\|_\e \leq \|\mathbf{\Sigma}(t)\varphi\|_\e \qquad \forall t \geq 0,\:\varphi \in \e_+.\end{equation}
Then, one gets the following
\begin{cor}\label{cor-nonhomo} For any $s \geq 0$ and any $f \in \e_+$
$$\U (t,s)f \in \D(\mathbf{B}(t)) \qquad \text{ for almost every } t \in [s,\infty)$$
and the mapping $t \in [s,\infty)\mapsto \mathbf{B}(t)\U(t,s)u$ is measurable. Moreover,
\begin{equation}\begin{split}\label{dissiKn}
\int_s^r \|\mathbf{B}(t)\U(t,s)f\|_{\e}\d t &\leq \int_s^r \|\mathbf{\Sigma}(t)\U(t,s)f\|_{\e}\d t \\
&\leq \|f\|_\e - \|\U(r,s)f\|_{\e} \qquad \forall \varphi\in \e_+\;,\,r\geq s \geq 0.\end{split}\end{equation}
\end{cor}
Again, the proof of this Corollary which uses techniques introduced in the proof of Lemma \ref{lem-technint} is given in Appendix \ref{app:integration}.\medskip

Notice that, as in the spatially homogeneous setting,  the family of perturbation operators $(\mathbf{B}(t),\D(\mathbf{B}(t)))_{t\geq 0}$ and the unperturbed evolution family $(\U(t,s))_{t\geq s}$ satisfy Assumptions \ref{hypo}.  In particular, Theorem \ref{generevol} applies and one can define an evolution family $\left(\V(t,s)\right)_{t\geq s\geq 0}$ in $\e$ that satisfies the Duhamel formula
\begin{equation}\label{V0tE}
\V(t,s)\varphi =\U(t,s)\varphi + \int_s^t \V(t,r)\mathbf{B}(r)\U(r,s)\varphi \d r,
\qquad \forall \varphi \in \e.
\end{equation}
In the space
$$\X=L^1(\mathbb{R}_+,\e)=L^1(\mathbb{R}_+ \times \mathbf{\Lambda},\d t\otimes\d x\otimes \d\mu(v))$$
one can then define  the $C_0$-semigroup $(\T_{0}(t))_{t\geq 0}$ given by \eqref{lien} and with generator $(\Z,\D(\Z))$, while, to the evolution family $\left(\V(t,s)\right)_{t\geq s\geq 0}$ one associates a $C_0$-semigroup $\Tt$ in $\X$ given by \eqref{lineTt}. Remember that the generator $(\G,\D(\G))$ of $\Tt$ is an extension of the operator $(\Z+\widehat{\fB},\D(\Z))$ where $\widehat{\fB}$ is constructed in Theorem \ref{construcB}. As in the spatially homogenous case, one can prove without major difficulty that, under our assumptions on $b(\cdot,\cdot,\cdot)$ and $\Sigma(\cdot,\cdot)$, Assumptions \ref{hypo1}, \ref{hypB} and \ref{hypo2} are also met, in particular, for any $g \in \D(\Z)\subset \X$, one has
$$\int_0^\infty \d t \int_{\mathbf{\Lambda}} \sigma(t,v)|g(t,x,v)|\d x \d\mu(v) < \infty$$
and
$$[\widehat{\fB}g](t,x,v)=[\mathbf{B}(t)g(t,x,\cdot)](v)=\int_V b(t,v,v')g(t,x,v')\d\mu(v') \qquad t \geq 0,(x,v) \in \mathbf{\Lambda}.$$

Now, to investigate the honesty of $\vt$, we shall, as in the autonomous case recently studied by the third author \cite{musAfrika}, use the honesty properties of the spatially homogeneous family $(\V_h(t,s))_{t\geq s\geq 0}$. To do so, let $\mathcal{M}$ be the spatial average operator
$$\mathcal{M}\::\:\varphi \in \e \mapsto \mathcal{M}\varphi(v)=\int_{\mathbf{\Omega}}\varphi(x,v)\d x, \qquad v \in V$$
one sees that $\mathcal{M}$ provides a \textit{bounded operator} from $\e$ to $\e_h:=L^1(V,\d\mu)$ and one has the following relation:
\begin{equation}\label{estimMM}\left[\mathcal{M}\U(t,s)\varphi\right](v)=\int_{\mathbf{\Omega}} \left[\U(t,s)\varphi\right](x,v)\d x \leq \left[\U_h(t,s)\mathcal{M}\varphi\right](v) \qquad \forall \varphi \in \e_+,\quad \text{ $\mu$-a. e. } v \in V\end{equation}
where $(\U_h(t,s))_{t \geq s \geq 0}$ is the evolution family in $L^1(V,\d\mu)$ introduced in the spatially homogeneous case (see \eqref{U0ts}).
From this, we deduce directly that
\begin{equation}\label{estimMM2}\|\U(t,s)\varphi\|_\e \leq \|\U_h(t,s)\mathcal{M}{\varphi}\|_{L^1(V,\d\mu)} \qquad \forall t \geq s \geq 0, \qquad \varphi \in \e_+\end{equation}  Since $\mathbf{B}(t)$ is local in $x$, one also has
$$\mathcal{M}\mathbf{B}(t)\varphi=\mathbf{B}_h(t)\mathcal{M}\varphi \qquad \forall \varphi \in \D(\mathbf{B}(t)), \qquad t \geq 0$$
where $\mathbf{B}_h(t)$ is the operator defined in \eqref{B0t} (one notices easily that, for any $\varphi \in \D(\mathbf{B}(t))$, $\mathcal{M}\varphi$ belongs to $\D(\mathbf{B}_h(t))$). These basic observations play  a crucial role in the study of autonomous transport equations as studied in \cite{musAfrika} and we shall see it is still the case here. Precisely,  the honesty of the evolution family $\vt$ can be deduced directly from that governing the spatially homogeneous problem introduced in the previous subsection. Namely,
\begin{theo} If the evolution family $(\V_h(t,s))_{t \geq s \geq 0}$ defined in  \eqref{V0t} is honest in $\e_h=L^1(V,\d\mu)$ then the evolution family $\vt$ is honest in $\e$.
\end{theo}
\begin{proof} For some $\varphi \in \e_+$, one deduces easily by induction from \eqref{estimMM} that, for any $n \in \mathbb{N}$,
\begin{equation}\label{average}
\left\|\mathbf{B}(t+r)\V_n(t+r,r)\varphi\right\|_\e \leq \left\|\mathbf{B}_h(t+r)\V_{n,h}(t+r,r)\mathcal{M}\varphi\right\|_{L^1(V,\d\mu)} \qquad \forall t,r \geq 0,\end{equation}
where $\V_{n,h}(t,r)$ is the $n$-th iterated term in the Dyson-Phillips expansion series giving $(\V_{h}(t,s))_{t \geq s\geq 0}$.
Since the latter is honest, Definition \ref{defi:hon} directly yields that $\vt$ is honest in $\e$.\end{proof}
\begin{nb} Recall that practical criteria yielding the honesty of $(\V_h(t,s))_{t \geq s \geq 0}$ in $\e_h=L^1(V,\d\mu)$ are provided in Theorem \ref{detaim}
\end{nb}
\begin{nb} Notice that, if $\mathbf{\Omega}=\mathbb{R}^d$, inequality \eqref{estimMM} is actually an identity. This implies in particular that inequality \eqref{average} is also an identity. Therefore, if $\mathbf{\Omega}=\mathbb{R}^d$, the evolution family $\vt$ is honest in $\e$ \emph{if and only if}  $(\V_h(t,s))_{t \geq s \geq 0}$ is honest in $\e_h$.
\end{nb}
\section{Application to non-autonomous fragmentation equation}\label{sec:frag}

We show in this section that the previous theoretical approach can be applied to the study of the non-autonomous fragmentation equation investigated recently in \cite{LB}. Precisely, let us consider the problem
\begin{equation}\label{fragment}
\left\{\begin{split}
\partial_t u(t,x)&=-a(t,x)u(t,x) + \int_{x}^\infty a(t,y)b(t,x,y)u(t,y)\d y  \qquad x >0, \:t > s\\
u(s,x)&=f(x) \quad \text{ a.e. } x >0
\end{split}\right.
\end{equation}
at some fixed $s \geq 0.$ As in \cite{LB}, $u$ is the distribution of particles with respect to mass $x$, $a(\cdot,\cdot)$ is the fragmentation rate and $b(t,x,y)$ is the
distribution of daughter particles of mass $x$ resulting from splitting of a parent particle of mass $y$. We investigate the above problem in the Banach space $$\e=L^1(\mathbb{R}^+, x\d x)$$
with usual norm which, for a non-negative function $u$ gives the total mass of the ensemble. We adopt here the following  set of assumptions:
\begin{hyp}\label{hyp:frag}
\begin{enumerate} The rates $a(\cdot,\cdot)$ and $b(\cdot,\cdot,\cdot)$ satisfy the following:
\item The mapping $(t,x) \in \mathbb{R}^+ \times \mathbb{R}^+ \mapsto a(t,x)$ is measurable and non-negative. Moreover, given $T > 0,$
for almost every $x \in \mathbb{R}^+$, the mapping $t \in [0,T] \mapsto a(t, x)$ is bounded.
\item The  mapping $(t, x, y) \in \mathbb{R}^+ \times \mathbb{R}^+ \times \mathbb{R}^+ \mapsto b(t, x, y)$ is measurable and non-negative. Moreover,
$b(t, x,y) = 0$ for almost every $x \geq y$ and $t >0$ with
$$\int_0^y b(t,x,y)x \d x = y \quad \text{ for any }  y \geq 0 \:\text{ and almost any } t \geq 0.$$
\end{enumerate}
\end{hyp}
We observe that the fragmentation equation \eqref{fragment} can be seen as a perturbation of the simpler non-autonomous advection equation
\begin{equation}\label{advection}
\left\{\begin{split}
\partial_t u(t,x)&=-a(t,x)u(t,x) \qquad x >0, \:t > s\\
u(s,x)&=f(x) \quad \text{ a.e. } x >0
\end{split}\right.
\end{equation}
One introduces the family $(\mathbf{A}(t))_{t \geq 0}$ of multiplication operators in $\e$ given by
$$[\mathbf{A}(t)u](x) = -a(t, x)u(x) \qquad  u  \in \D(\mathbf{A}(t))$$
with domain $\D(\mathbf{A}(t))$ given by
$$\D(\mathbf{A}(t))=\{u \in \e\;;\;\mathbf{A}(t)u \in \e\}=L^1(\mathbb{R}^+,\,a(t,x)x\d x).$$ Let then define the evolution family in $\e$ given by
$$\U(t, s)u = \exp\left(-\int_s^t a(\t,\cdot)\d \t\right)u(\cdot) \qquad \:t \geq s \geq 0, \qquad u \in \e$$
and the  family of perturbations $(\mathbf{B}(t))_{t \geq 0}$ given by
$$\left[\mathbf{B}(t)u\right](x)=\int_{x}^\infty a(t,y)b(t,x,y)u(t,y)\d y \qquad \forall u \in \D(\mathbf{B}(t))=\D(\mathbf{A}(t)), \quad t \geq 0.$$
 It is easy to verify, using Assumptions \ref{hyp:frag} and arguing as in the previous section, that the evolution family $(\U(t, s))_{t \geq s}$ and the family
of perturbations $(\mathbf{B}(t))_{t \geq 0}$ satisfy Assumptions \ref{hypo}, \ref{hypo1}, \ref{hypB} and \ref{hypo2}. In particular one can see that, for any
$u \in \e$ and any $s \geq 0$ one has $\U(t, s)u \in \D(\mathbf{B}(t))$ for almost every $t \geq s$; the mapping $t \in [s;\infty) \mapsto \mathbf{B}(t)\U(t, s)u \in \e$ is measurable with
$$\int_s^t \|\mathbf{B}(t)\U(t, s)u \| \d s = \|u\|_\e - \|\U(t,s)u\|_\e  \qquad \forall t \geq s \,\quad u \in \e_+.$$
These conditions allow to apply here the general abstract approach developed in Sections 2, 3.1 and 3.2. We have then, for any $u \in \e$ and any $t \geq s \geq 0$,
$$\V_0(t, s)u = \U(t, s)u,$$
and, for $n \geq 1$
\begin{equation}\label{vnfrag}
\V_n(t, s)u =\int_s^t \U(t, r)\mathbf{B}(r)\V_{n-1}(r, s)u \d r\end{equation} since, for any $n \geq 0$, $s \geq 0,$ and  $u \in \e$ one has $\V_n(t,s)u \in \D(\mathbf{B}(t))$ for almost every $t \geq s$, and the mapping $t \in [s;\infty) \mapsto
\mathbf{B}(t)\V_n(t, s)u \in \e $ is measurable.
In the present case we can also prove the following:
\begin{propo} For any $n \in \mathbb{N}$, any $s \geq 0$ and any $u \in \e_+$,  the mapping $r \in [s,\infty) \mapsto \mathbf{A}(r)\V_n(r,s)u \in \e$ is measurable; furthermore for
any $t \geq  s \geq 0$ and any $u \in \e$ one has
\begin{equation}\label{eq1}\U(t, s)u = u +
\int_s^t \mathbf{A}(r)\U(r, s)u \d r\end{equation}
and, for $n \geq 1$
\begin{equation}\label{eq2}\V_n(t, s)u =
\int_s^t \mathbf{A}(r)\V_n(r, s)u \d r +
\int_s^t \mathbf{B}(r)\V_{n-1}(r, s)u \d r.\end{equation}
\end{propo}
\begin{proof} The proof of \eqref{eq1} is almost a straightforward application of the definition of $\mathbf{A}(t)$ and $\U(t,s)$ (see for instance \cite[Lemma 2.1]{LB}). Now, for $n \geq 1$, using both \eqref{vnfrag} and \eqref{eq1}, we get
$$\V_n(t,s)u=\int_s^t \mathbf{B}(r)\V_{n-1}(r,s)u \d r  + \int_s^t \d r\int_r^t \mathbf{A}(\t)\U(\t,r)\mathbf{B}(r)\V_{n-1}(r,s)u\d \t$$
for any $u \in \e$ and any $t \geq s\geq 0.$ But, using Assumption \ref{hypo2} and \eqref{eq2}, we check easily that
\begin{multline*}
\int_s^t \d r\int_r^t \mathbf{A}(\t)\U(\t,r)\mathbf{B}(r)\V_{n-1}(r,s)u\d \t=\int_s^t \d \t\int_s^\t \mathbf{A}(\t)\U(\t,r)\mathbf{B}(r)\V_{n-1}(r,s)u\d r\\
=\int_s^t \mathbf{A}(\t)\left(\int_s^\t \U(\t,r)\mathbf{B}(r)\V_{n-1}(r,s)u\d r\right)\d \t =\int_s^t \mathbf{A}(\t)\V_n(\t,s)\d\t\end{multline*}
from which the conclusion easily follows.
\end{proof}

Thanks to the above proposition we obtain
\begin{propo}
Let $u \in \e$ and $t \geq  s \geq 0$ be given. For any $n \in \mathbb{N}$, let $v_n(t, s,\cdot)$ denote   a measurable representation of
$\sum_{k=0}^n
[\V_k(t, s)u](\cdot)$, then for almost every $x  \geq 0$ one has
$$v_n(t, s, x) = u(x) -\int_s^ta(r, x)v_n(r, s, x)\d r +
\int_s^t \int_x^\infty  a(r, y)b(r, x, y)v_{n-1}(r, s, y)\d y\d r.$$
\end{propo}
\begin{proof} The identities \eqref{eq1} and \eqref{eq2} give
$$\sum_{k=0}^n
\V_k(t, s)u= u +
\int_s^t
\mathbf{A}(r)
\sum_{k=0}^n
\V_k(r, s)u\d r +
\int_s^t
\mathbf{B}(r)
\sum_{k=0}^{n-1}\V_k(r, s)u\d r,$$
and this gives the assertion. \end{proof}
Now suppose $v(t, s, \cdot)$ to be a measurable representation of $\V(t,s)u$. According to Assumptions \ref{hyp:frag} (1), for any
$t \geq s \geq 0$ one has
$$\int_s^t
 a(r, x)v(r, s, x)\d r < \infty \qquad  \text{for a. e. } x \geq 0.$$
Therefore we also have
$$\int_s^t
\d r\int_x^\infty
a(r, y)b(r, x, y)v(r, s, y)\d y < \infty \quad \text{for a. e. } x \geq 0$$  and the following equality holds:
$$v(t,s, x) = u(x) -\int_s^t
a(r, x)v(r, s, x)\d r +
\int_s^t
\d r\int_x^\infty
a(r, y)b(r, x, y)v(r, s, y)\d y.$$
This allows us to recover, in a direct way,  Theorem 4.1 of \cite{LB}.
Moreover, because the honesty of the evolution family $\vt$ is equivalent to that of the evolution semigroup $\Tt$, if one assumes that
$$a(t,x) \in L^1(0,T\,;\,L^\infty([0,R])) \qquad \text{ for any } T >0, R >0,$$
 we have, according to \cite[Theorem 5.1]{LB}, that  $\vt$ is honest in $\e$. This implies that the limit \eqref{eq:defihonest} holds.
\begin{nb}
Notice that, using the results of Sections 3 and 4, we did not need to determine the generator of $\Tt$ to recover \cite[Theorem 4.1 \& Theorem 5.1]{LB} and this allows therefore to simplify the approach of \cite{LB}. Notice also that the fact that \eqref{eq:defihonest} holds for the fragmentation equation \eqref{fragment} was not observed in \cite{LB} and is therefore \emph{a new result}.
\end{nb}

 \appendix

\section{Technical properties of $\Z$ and $\widehat{\fB}$}

We derive in this Appendix several fine properties of the operators $\Z$ and $\widehat{\fB}$ introduced in Section 3. The notations and Assumptions are those of Section 3. In particular, we assume that Assumptions \ref{hypo1} and \ref{hypB} are in force. We begin with
\begin{lemme}\label{lemmeAvarphis} Given $s >0$  and $\varphi \in  \mathscr{C}_c^1(\mathbb{R}_+)$, let us define the operator
\begin{equation}\begin{cases}\label{Avarphis}
\mathcal{A}_{\varphi,s}\::\:u \in \e &\longmapsto  \mathcal{A}_{\varphi,s}u \in \X\\
&\text{ with } \quad [\mathcal{A}_{\varphi,s}u](t)=\chi_+(t-s)\varphi(t)\U(t,s)u \qquad \forall t\geq 0.\end{cases}\end{equation}
Then $\mathcal{A}_{\varphi,s}u \in \D(\Z)$ for any $u \in \e$ with
$$\Z \mathcal{A}_{\varphi,s}=-\mathcal{A}_{\varphi',s}$$
where $\varphi'$ is the derivative of $\varphi$. Moreover,
\begin{equation}\label{fBAs}[\widehat{\fB}\mathcal{A}_{\varphi,s}u](t)=\mathbf{B}(t)[\mathcal{A}_{\varphi,s}u](t) \qquad \text{ for almost every } t \geq 0,\end{equation}
i.e. $\mathcal{A}_{\varphi,s}u\in \D(\fB)$ and $\widehat{\fB}\mathcal{A}_{\varphi,s}u=\fB \mathcal{A}_{\varphi,s}u$.
\end{lemme}
\begin{proof} Let $s >0$, $u \in  \e$
and $\varphi\in  \mathscr{C}_c^1(\mathbb{R}_+)$. For simplicity, we set $f_0=\mathcal{A}_{\varphi,s}u,$ i.e.
$f_0(\cdot)=\chi_+(\cdot - s)\varphi(\cdot)\U(\cdot,s)u$. Clearly,
$$\|f_0\|_\X=\int_0^\infty \|f_0(t)\|_\e \d t=\int_s^\infty |\varphi(t)|\,\|\U(t,s)u\|_\e\d s \leq 2\gamma\|u\|_\e \int_s^\infty |\varphi(t)|\d t$$
and, since $\varphi$ is compactly supported, $f_0 \in \X$.
 By \eqref{lien}, one has
$$\T_0(t)f_0(\cdot)=\chi_{+}(\cdot-t)\U(\cdot,\cdot-t)f_0(\cdot-t)$$
and one checks easily that
\begin{equation*}
\T_0(t)f_0(r)=\chi_{+}(r-t-s)\varphi(r-t)\U(r,s)u \qquad \forall r >0.\end{equation*}
Clearly, for given $r >0$,
$$\lim_{t \to 0^+} t^{-1}\bigg(\chi_+(r-t-s)\varphi(r-t)-\chi_+(r-s)\varphi(r)\bigg)=-\chi_+(r-s)\varphi'(r)$$ and, since $\varphi$ is compactly supported, one deduces easily that the following limit holds in $\X$:
$$\lim_{t \to 0^+} t^{-1}\left(\T_0(t)f_0(\cdot)-f_0(\cdot)\right)=-\chi_+(\cdot-s)\varphi'(\cdot)\U(\cdot,s)u.$$
This exactly means that $f_0 \in \D(\Z)$ with $\Z f_0(\cdot)=-\chi_+(\cdot-s)\varphi'(\cdot)\U(\cdot,s)u$, i.e. $\Z \mathcal{A}_{\varphi,s}=-\mathcal{A}_{\varphi',s}$. Let us now prove that $f_0 \in \D(\fB)$. Since, for any $t >0$ and any $\tau \in (0,t)$, one has as above
$$\U(t,\tau)f_0(\tau)=\chi_+(\tau-s)\varphi(\tau)\U(t,s)u$$
one deduces from  \eqref{represfBl} that, for any $\l >0$,
$$[\fB_\l f_0](t)=\left(\int_s^t \exp(-\l(t-\t))\varphi(\tau)\d\tau\right)\chi_+(t-s)\mathbf{B}(t)\U(t,s)u \quad \text{ for  every  } t > s$$
where we recall that $\mathbf{B}(t)\U(t,s)u$ is well-defined for almost any $t \geq s.$ Then, since
$$\lim_{\l \to \infty} \l\left(\int_s^t \exp(-\l(t-\t))\varphi(\tau)\d\tau\right)=\varphi(t) \quad \text{ for almost every  } t \geq s \geq 0$$
we get easily that the following limit holds in $\X$:
$$\lim_{\l \to \infty} \l \fB_\l f_0(\cdot)=\chi_+(\cdot-s)\varphi(\cdot)\mathbf{B}(\cdot)\U(\cdot,s)u=\mathbf{B}(\cdot)f_0(\cdot).$$
This proves in particular that $f_0 \in \D(\fB)$ and, since $\lim_{\l \to \infty}\l \fB_\l=\widehat{\fB}$ we get $\widehat{\fB}f_0=\fB f_0$ which achieves the proof.
\end{proof}

\begin{lemme}\label{lem:AppA1}
Given $\t >0$ and $f \in \X$
$$\mathbf{B}(s)\U(s,\t)\int_0^\t \U(\t,r)f(r)\d r=\int_0^\t \mathbf{B}(s)\U(s,r)f(r)\d r \qquad \text{ for almost every } s > \t.$$
In particular, $\int_0^\t \U(s,r)f(r)\d r \in \D(\mathbf{B}(s))$ for almost every $s >\t$ with
$$\mathbf{B}(s)\int_0^\t \U(s,r)f(r)\d r=\int_0^\t \mathbf{B}(s)\U(s,r)f(r)\d r \qquad \text{ for almost every } s > \t.$$
\end{lemme}
\begin{proof} For any $f \in \X_+$ and any fixed $\t >0$, we introduce $g_\t=\int_0^\t \U(\t,r)f(r)\d r \in \e_+.$ According to Assumptions \ref{hypo}, the mapping $ s \in \mathbb{R}^+ \mapsto \chi_+(s-\t)\mathbf{B}(s)\U(s,\t)g_\t$ is measurable and it belongs to $\X$ since
\begin{equation}\label{eq:ine1}\int_\t^\infty \|\mathbf{B}(s)\U(s,\t)g_\t\|_\e \leq \| g_\t\|_\e \leq \|f\|_\X.\end{equation}
Now, thanks to Assumptions \ref{hypo1}, the mapping $(s,r) \in \mathbb{R}^+ \times \mathbb{R}^+ \mapsto \chi_+(s-r)\mathbf{B}(s)\U(s,r)f(r)$ is integrable so that, for almost every $s >0$, the integral $\int_0^s \mathbf{B}(s)\U(s,r)f(r)\d r$ exists and the mapping $s > 0 \mapsto \int_0^s \mathbf{B}(s)\U(s,r)f(r)\d r \in \e_+$ is integrable. In particular, for some fixed $\t >0$, the integral $\int_0^\t \mathbf{B}(s)\U(s,r)f(r)\d r \in \e_+$ exists for almost every $s > \t$ and
\begin{multline}\label{eq:ine}\int_\t^\infty \left\|\int_0^\t \mathbf{B}(s)\U(s,r)f(r)\d r\right\|_\e \d s =\int_\t^\infty \left(\int_0^\t \|\mathbf{B}(s)\U(s,r)f(r)\|_\e \d r\right)\d s\\
\leq \int_{\Delta} \|\mathbf{B}(s)\U(s,r)f(r)\|_\e \d r \d s \leq \|f\|_\X.
\end{multline}
Thus  the mapping $G_\t\;:\; s \in \mathbb{R}^+ \mapsto \chi_+(s-\t)\int_0^\t \mathbf{B}(s)\U(s,r)f(r)\d r$ belongs to $\X$ with $\|G_\t\|_\X \leq \|f\|_\X.$ Let now $\mathfrak{D}$ be the subspace of
$\X$ made of  all linear combinations  of elements of the form: $$t >0 \longmapsto
\varphi(t)\U(t,s)u\,\qquad \text{ where } \,s >0\,,\, u\in
\e$$ and $\varphi \in \mathscr{C}_c^1(\mathbb{R})$ with $\mathrm{supp}(\varphi) \in
[s,\infty)$. Then, for $f \in \mathfrak{D}$, it is easily seen  that
$$\mathbf{B}(s)\U(s,\t)g_\t=\int_0^\t \mathbf{B}(s)\U(s,r)f(r)\d r \qquad \text{ for almost every } s > \t$$
i.e. the two functions represent the same element of $\X$. Since $\mathfrak{D}$ is dense in $\X$ \cite[Remark 1.11]{liske}, the above inequalities \eqref{eq:ine1} and \eqref{eq:ine} allow  to assert that the above identity extends to any $f \in \X_+$ proving the first part of the Lemma. Since moreover
$$\U(s,\t)\int_0^\t \U(\t,r)f(r)\d r=\int_0^\t \U(s,r)f(r)\d r$$
we get that, for almost every $s >\t$, the integral $\int_0^\t \U(s,r)f(r)\d r$ belongs to $\D(\mathbf{B}(s))$  and the conclusion follows.\end{proof}

We complement the above Lemma with the following
\begin{lemme}\label{lem:AppA2}
For any $f \in \X$ and and almost all $(s,\t) \in \Delta$, one has $\int_0^\t \U(s,r)f(r)\d r \in \D(\mathbf{B}(s))$  and
$$\mathbf{B}(s)\int_0^\t \U(s,r)f(r)\d r=\int_0^\t \mathbf{B}(s)\U(s,r)f(r)\d r.$$
\end{lemme}
\begin{proof} For some fixed  $f \in \X$, we notice that the two functions
$$H_f\::\:(s,\t) \in \Delta \mapsto \mathbf{B}(s)\U(s,\t)\int_0^\t \U(\t,r)f(r)\d r \in \e $$ and
$$\tilde{H}_f\::\:(s,\t) \in \Delta \mapsto \int_0^\t \mathbf{B}(s)\U(s,r)f(r)\d r$$
are measurable. Therefore, the subset $\Delta_0=\{(s,\t)\in\Delta\,;\,H_f(s,\t)\neq \tilde{H}_f(s,\t)\}$ is a measurable subset of $\Delta$. Since, given $\t >0$, one has $H_f(s,\t)=\tilde{H}_f(s,\t)$ for almost every $s \geq \t$, one sees that $\Delta_0$ is a set of zero measure. Since moreover
$$\U(s,\t)\int_0^\t \U(\t,r)f(r)\d r=\int_0^\t \U(s,r)f(r)\d r$$
we get that, for almost every $s >0$, the integral $\int_0^\t \U(s,r)f(r)\d r$ belongs to $\D(\mathbf{B}(s))$ for almost every $\t \in (0,s)$ with the desired conclusion.
\end{proof}

\begin{lemme}\label{lem:AppA3}
For any $f \in \X_+$ and any $t >0$, set $g_t:=\int_0^t \T_0(s)f \d s \in \D(\Z)_+.$ One has
$$\la \mathbf{\Psi}\,,\widehat{\fB}g_t \ra_\X=\int_0^\infty \left(\int_{\t}^{\t+t} \|\mathbf{B}(r)\U(r,\t)f(\t)\|_\e \d r\right)\d \t.$$
\end{lemme}
\begin{proof}   One has
\begin{multline*}
\left[\widehat{\fB} g_t\right](r)=\mathbf{B}(r)g_t(r)=\mathbf{B}(r)\int_0^t \left[\T_0(s)f\right](r)\d s\\
=\mathbf{B}(r)\int_0^t \U(r,r-s)f(r-s)\chi_+(r-s)\d s.
\end{multline*}
In particular, according to Lemma \ref{lem:AppA2} and Proposition \ref{prop:intDB}, for almost any $t > 0$, one has
$$\left[\widehat{\fB} g_t\right](r)=\int_{(r-t)\vee 0}^r \mathbf{B}(r)\U(r,\t)f(\t)\d\t \qquad \text{ for almost any } r >0.$$
Therefore, for almost every $t > 0$,
\begin{equation}\label{identBgtA}\begin{split}
\la \mathbf{\Psi}\,,\widehat{\fB}g_t \ra_\X&=\int_0^\infty \d r \int_{(r-t)\vee 0}^r \left\|\mathbf{B}(r)\U(r,\t)f(\t)\right\|_\e \d\t\\
&=\int_0^\infty \left(\int_{\t}^{\t+t} \|\mathbf{B}(r)\U(r,\t)f(\t)\|_\e \d r\right)\d \t.\end{split}\end{equation}
Since both the right-hand side and the left-hand side of \eqref{identBgtA} are continuous functions of $t$, the above actually holds \emph{for any} $t >0.$
\end{proof}

We end this Appendix with a basic observation, based upon Lemma \ref{lem:AppA2}:
\begin{lemme}\label{lemBlexp} For any $\l >0$ and any $f \in \X_+$,
$$[\fB_\l \exp(-\l t)\T_0(t) f](s)=\mathbf{B}(s)\int_0^{s-t} \exp(-\l(s-\tau))\U(s,\tau)f(\tau)\d\tau \qquad \text{ for a. e. } (s,t) \in \Delta.$$
\end{lemme}
\begin{proof} The proof of the result is an almost straightforward application of Lemma \ref{lem:AppA2}. Namely, for a given $t \geq 0$, one has
\begin{multline*}
[\fB_\l \T_0(t)f](s)=\int_0^s \exp(-\l(s-\tau))\mathbf{B}(s)\U(s,\tau)\left[\T_0(t)f\right](\tau)\d\tau\\
=\int_0^{s-t}\exp(-\l r)\mathbf{B}(s)\U(s,s-t)\U(s-t,s-t-r)f(s-t-r)\d r \qquad \text{ for a. e. } s \geq t\end{multline*}
where we used \eqref{lien} and the measurability assumption \ref{hypo1}. Now, by Lemma \ref{lem:AppA2}, we get
easily the conclusion.\end{proof}

\section{Technical integration results for kinetic equations}\label{app:integration}

This last Appendix is devoted to the proof of the technical Lemma \ref{lem-technint} and its consequence, Corollary \ref{cor-nonhomo}. We first introduce some useful notations:  let
$$\mathbf{\Lambda}_\pm:=\{(x,v)\in \mathbf{\Lambda}\,;\;\tau(x,\pm v) < \infty\}, \qquad \mathbf{\Lambda}_{\pm,\infty}:=\{(x,v)\in \mathbf{\Lambda}\;;\;\tau(x,\pm v)=+\infty\}$$
$$\mathbf{\Gamma}_-:=\{(y,v) \in \partial \mathbf{\Omega} \times V\;\text{ such that } y=x-\tau(x,v)v \text{ for some } x \in \mathbf{\Omega} \text{ with } \tau(x,v) <\infty\}$$
$$\mathbf{\Gamma}_+:=\{(y,v) \in \partial \mathbf{\Omega} \times V\;\text{ such that } y=x+\tau(x,-v)v \text{ for some } x \in \mathbf{\Omega} \text{ with } \tau(x,-v) <\infty\}.$$
Finally, set
$$\mathbf{\Gamma}_{\pm,\infty}:=\{(y,v) \in \mathbf{\Gamma}_{\pm}\;;\;y\mp t v \in \mathbf{\Omega} \;\;\forall t > 0\}.$$

\begin{nb}\label{taumeasu} From a heuristic viewpoint, $\tau(x,v)$ is the time needed by a particle having the position
$x$ and the velocity $v \in V$ to reach the boundary $\partial \mathbf{\Omega}$. One can prove  that $\tau(\cdot,\cdot)$ is measurable on $\mathbf{\Lambda}$. Moreover $\tau(x,v) = 0$ for any $(x, v) \in \mathbf{\Gamma}_-$ whereas $\tau(x, v) > 0$ on $\mathbf{\Gamma}_+.$  Notice that the mapping
$$\tau\::\:(x,v) \in \mathbf{\Lambda} \longmapsto  \tau(x,v)$$
is lower semi-continuous \cite{jacek} (see also \cite{voigtHabi}). In particular,
$$\mathbf{\Lambda}_{\pm,\infty}=\bigcap_{n \in \mathbb{N}}\left\{(x,v) \in \mathbf{\Lambda}\,;\;\tau(x,\pm v) > n\right\}$$
is the numerable intersections of open subsets of $\mathbb{R}^{2d}.$ Therefore, $\mathbf{\Lambda}_{\pm,\infty}$  and $\mathbf{\Lambda}_{-,\infty} \cap \mathbf{\Lambda}_{+,\infty}$ are Borel subsets of $\mathbf{\Lambda}$.

\end{nb}

Let us then introduce the state space
$$\e=L^1(\mathbf{\Lambda}\,,\,\d x \otimes \d\mu(v))$$
with usual norm. One recalls then the following, from \cite[Proposition 2.5 \& Proposition 2.10]{jacek},
\begin{lemme}  There exists Borel measures $\nu_{\pm}$ on $\mathbf{\Gamma}_\pm$ such that, for any $f \in L^1(\O,\d \mu)$, one has
\begin{equation}\label{10.47}
\int_{\mathbf{\Lambda}_{\pm}}f(x,v)\d\mu(v)\d x
=\int_{\mathbf{\Gamma}_\pm}\d\nu_\pm(y,v)\int_0^{\tau(y,\pm v)}f(y\mp
sv,v)\d s,
\end{equation}
and
\begin{equation}\label{10.49}
\int_{\mathbf{\Lambda}_{\pm} \cap
\mathbf{\Lambda}_{\mp,\infty}}f(x,v)\d\mu(v)\d x=\int_{\mathbf{\Gamma}_{\pm,\infty}}\d\nu_\pm(y,v)\int_0^{\infty}f\left(y\mp
sv,v\right)\d s.
\end{equation}
\end{lemme}
\begin{proof} The proof follows from the general theory developed in \cite[Propositions 2.5 \& 2.10]{jacek} where, with the notations of \cite{jacek}, one has
$$\mathcal{F}(\mathbf{\x})=(v,0) \qquad \forall \mathbf{\x}=(x,v)\in \mathbf{\Lambda}$$
which is divergence free in the sense of \cite{jacek} and the characteristic curves are given by $\zeta(\mathbf{x},t,s)=(x-(t-s)v,v)$ for any $t,s \in \R,$ $\mathbf{x}=(x,v) \in \mathbf{\Lambda}.$
\end{proof}
Let us then introduce  the  unperturbed evolution family $(\U(t,s))_{t \geq s\geq 0}$ defined through
$$[\U(t,s)f](x,v)=\exp\left(-\int_s^t \Sigma(v,\tau)\d\tau\right)f(x-(t-s)v,v)\chi_{\{t-s < \tau(x,v)\}}(x,v),$$
for any $t\geq s \geq 0$ and any $f \in \e$.  One has the following
\begin{lemme} For any $f \in \e$ and any $t \geq s\geq 0$, setting
$$\Theta(t,s,v)=\exp\left(-\int_s^t \Sigma(r,v)\d r\right) \qquad \forall v \in V$$
one has
\begin{multline}\label{U1t}
\int_{\mathbf{\Lambda}_-}\left[\U(t,s)f\right](x,v)\d x\d\mu(v)=\\
\int_{\mathbf{\Gamma}_-}\d\nu_-(y,v)\left(\int_0^{0 \vee (\tau(y,-v)-(t-s))} \Theta(t,s,v)f(y+rv,v)\d r\right)
\end{multline}
\begin{multline}\label{U2t}
\int_{\mathbf{\Lambda}_+ \cap \mathbf{\Lambda}_{-,\infty}}\left[\U(t,s)f\right](x,v)\d x\d\mu(v)=\\
\int_{\mathbf{\Gamma}_{+,\infty}}\left(\int_{t-s}^{\infty} \Theta(t,s,v)f(y-rv,v)\d r\right)\d\nu_+(y,v)
\end{multline}
and
\begin{multline}\label{U3t}
\int_{\mathbf{\Lambda}_{+,\infty} \cap \mathbf{\Lambda}_-}\left[\U(t,s)f\right](x,v)\d x\d\mu(v)=\\
\int_{\mathbf{\Gamma}_{-,\infty}}\left(\int_{t-s}^\infty \Theta(t,s,v)f(y+rv,v)\d r\right)\d\nu_-(y,v).
\end{multline}
\end{lemme}
\begin{proof} We prove only \eqref{U1t}, the other identities being similar. According to \eqref{10.47}
\begin{multline}\label{U1t1}
\int_{\mathbf{\Lambda}_-}\left[\U(t,s)f\right](x,v)\d x\d\mu(v)=\\
\int_{\mathbf{\Gamma}_-}\d\nu_-(y,v)\left(\int_0^{\tau(y,-v)} \left[\U(t,s)f\right](y+r_0 v,v)\d r_0\right).
\end{multline}
Now, for a given $(y,v)\in \mathbf{\Gamma}_-$ and $0 \leq r_0 \leq \tau(y,-v)$, since $\tau(y+r_0 v,v)=\tau(y,v)+r_0=r_0$ one has
\begin{multline*}
[\U(t,s)f](y+r_0v,v)=\Theta(t,s,v)f(y+r_0 v-(t-s)v,v)\chi_{\{t-s \leq \tau(y+r_0 v,v)\}}(y,v)\\
=\Theta(t,s,v)f(y-(t-r_0-s)v,v)\chi_{\{t-s  \leq r_0\}}(y,v).
\end{multline*}
Integrating with respect to $r_0$ and making the change of variable $r_0 \to r=r_0-(t-s)$, we get the result.\end{proof}

\begin{nb}\label{nbSplit} It is natural to identify $L^1(\mathbf{\Lambda}_-)$, $L^1(\mathbf{\Lambda}_{+}\cap \mathbf{\Lambda}_{-,\infty})$ and $L^1(\mathbf{\Lambda}_{+,\infty}\cap \mathbf{\Lambda}_{-,\infty})$ to closed linear subspaces of $\e=L^1(\mathbf{\Lambda})$. With such an identification, one has
$$\|f\|_\e=\|f_1\|_{L^1(\mathbf{\Lambda}_-)} + \|f_2\|_{L^1(\mathbf{\Lambda}_+ \cap \mathbf{\Lambda}_{-,\infty})} + \|f_3\|_{L^1(\mathbf{\Lambda}_{+,\infty}\cap \mathbf{\Lambda}_{-,\infty})} \qquad \forall f \in \e$$
where $f_1,f_2,f_3$ denote the restrictions of $f$ to $\mathbf{\Lambda}_-$, $\mathbf{\Lambda}_{+}\cap \mathbf{\Lambda}_{-,\infty}$ and $\mathbf{\Lambda}_{+,\infty}\cap\mathbf{\Lambda}_{-,\infty}$ respectively. Moreover, it is easy to see that $\U(t,s)$ leaves invariant the three subspaces  $L^1(\mathbf{\Lambda}_-)$, $L^1(\mathbf{\Lambda}_{+}\cap \mathbf{\Lambda}_{-,\infty})$ and $L^1(\mathbf{\Lambda}_{+,\infty}\cap \mathbf{\Lambda}_{-,\infty})$ in the following sense: with the above decomposition $\mathrm{Supp}(\U(t,s)f_1) \subset \mathbf{\Lambda}_-$, $\mathrm{Supp}(\U(t,s)f_2) \subset \mathbf{\Lambda}_{+}\cap \mathbf{\Lambda}_{-,\infty}$ and $\mathrm{Supp}(\U(t,s)f_3) \subset \mathbf{\Lambda}_{+,\infty}\cap \mathbf{\Lambda}_{-,\infty}$ for any $t\geq s\geq 0.$ Therefore, to prove any result on $\U(t,s)f$ for some given $f \in \e$, it is equivalent to consider separately $\U(t,s)f_i$, $i=1,2,3$ in their respective spaces.\end{nb}

One can now prove Lemma \ref{lem-technint}.

\begin{proof}[Proof of Lemma \ref{lem-technint}] According to Remark \ref{nbSplit}, one can distinguish the three cases $f \in L^1(\mathbf{\Lambda}_-)$, $f \in L^1(\mathbf{\Lambda}_{+}\cap \mathbf{\Lambda}_{-,\infty})$ and $f \in L^1(\mathbf{\Lambda}_{+,\infty}\cap\mathbf{\Lambda}_{-,\infty})$.\\

\noindent \textit{First case:} Let us assume that $f \in L^1(\mathbf{\Lambda}_-)$ and let $s \geq 0$ be fixed. One first assumes that $f$ is continuous. For some given $(y,v) \in \mathbf{\Gamma}_-$, one considers the continuous functions
$$\Q_{y,v}\::\:t \in [s,\infty) \longmapsto  \int_0^{0 \vee (\tau(y,-v)-(t-s))} \Sigma(t,v) \Theta(t,s,v)f(y+rv,v)\d r.$$
and
$$H_{y,v}\::\:t \in [s,\infty) \longmapsto  \int_0^{0 \vee (\tau(y,-v)-(t-s))} \Theta(t,s,v)f(y+rv,v)\d r.$$
Notice that $\Q_{y,v}(t)=0=H_{y,v}(t)$ for $t \geq \tau(y,-v)+s.$ Let us compute the derivative of $H_{y,v}(t)$. Clearly,
$$\dfrac{\d}{\d t}H_{y,v}(t)=0 \qquad t \in [\tau(y,-v)+s,\infty).$$
Let us now consider $t \in (s,\tau(y,-v)+s).$ One has
\begin{multline*}
\dfrac{\d}{\d t}H_{y,v}(t)=\dfrac{\d}{\d t}\int_0^{\tau(y,-v)-(t-s)} \Theta(t,s,v)f(y+rv,v)\d r\\
=\int_0^{\tau(y,-v)-(t-s)} \left(\dfrac{\partial}{\partial t} \Theta(t,s,v)\right)f(y+rv,v)\d r
- \Theta(t,s,v)f(y+(\tau(y,-v)-(t-s))v,v).
\end{multline*}
Now, denoting as above $\mathcal{L}_{\Sigma(\cdot,v)}$ as the set of Lebesgue points of the mapping $t \mapsto \Sigma(t,v)$, one notices that
\begin{equation}\label{deriveeTheta}\dfrac{\partial}{\partial t}\Theta(t,s,v)=-\Sigma(t,v)\Theta(t,s,v) \qquad \forall t \in \mathcal{L}_{\Sigma(\cdot,v)}\end{equation}
and, for such $t \in \mathcal{L}_{\Sigma(\cdot,v)}$, one has
\begin{multline*}\dfrac{\d}{\d t}H_{y,v}(t)=-\int_0^{\tau(y,-v)-(t-s)} \Sigma(t,v)\Theta(t,s,v)f(y+rv,v)\d r \\
- \Theta(t,s,v)f(y+(\tau(y,-v)-(t-s))v,v).\end{multline*}
In other words, $H_{y,v}$ is derivable at almost every $t \in [s,\infty)$ and, since $f$ is nonnegative, one has
\begin{equation}\label{eq:cas1}
\dfrac{\d}{\d t}H_{y,v}(t) \leq -{\Q}_{y,v}(t) \qquad \text{ for almost every } t \geq s.\end{equation}
Integrating this inequality with respect to $t$, one finds
$$\int_s^\infty  {\Q}_{y,v}(t)\d t \leq -\int_s^\infty \dfrac{\d}{\d t}H_{y,v}(t)\d t=H_{y,v}(s)=
\int_0^{\tau(y,-v)}f(y+rv,v)\d r.$$
Integrating now over $\mathbf{\Gamma}_-$, we find
$$\int_s^\infty \d t \int_{\mathbf{\Gamma}_-}\Q_{y,v}(t)\d\nu_-(t,v) \leq \int_{\mathbf{\Gamma}_-}\d\nu_-(y,v)\int_0^{\tau(y,-v)}f(y+rv,v)\d r=\|f\|_{L^1(\mathbf{\Lambda}_-)}.$$
In other words, recalling the expression of $\Q_{y,v}$, one gets that the linear form
\begin{multline*}f \in L^1(\mathbf{\Lambda}_-) \cap \mathscr{C}(\mathbf{\Lambda}_-) \longmapsto \int_s^\infty\d t\int_{\mathbf{\Gamma}_-}\d\nu_-(y,v)\\
\int_0^{0 \vee (\tau(y,-v)-(t-s))} \Sigma(t,v) \Theta(t,s,v)f(y+rv,v)\d r\end{multline*}
extends in an unique way to a continuous linear form on $L^1(\mathbf{\Lambda}_-)$ such that, for any nonnegative $f \in L^1(\mathbf{\Lambda}_-)$,
$$\int_s^\infty\d t\int_{\mathbf{\Gamma}_-}\d\nu_-(y,v)\int_0^{0 \vee (\tau(y,-v)-(t-s))} \Sigma(t,v) \Theta(t,s,v)f(y+rv,v)\d r \leq \|f\|_{L^1(\mathbf{\Lambda}_-)}.$$
In particular, for almost every $t \in [s,\infty)$, one has
$$\int_{\mathbf{\Gamma}_-}\d\nu_-(y,v)\int_0^{0 \vee (\tau(y,-v)-(t-s))} \Sigma(t,v) \Theta(t,s,v)f(y+rv,v)\d r < \infty$$
According to \eqref{U1t}, this exactly  means that
$$\int_{\mathbf{\Lambda}_-}\Sigma(t,v)[\U(t,s)f](x,v)\d\mu(v)\d x < \infty \qquad \text{ for almost every } t \in [s,\infty)$$
and yields the desired result whenever $f \in L^1(\mathbf{\Lambda}_-).$\\

\noindent \textit{Second case:} Let us assume that $f \in L^1(\mathbf{\Lambda}_{-,\infty} \cap \mathbf{\Lambda}_+)$ and let $s \geq 0$ be fixed and let $(y,v) \in \mathbf{\Gamma}_{+,\infty}$ be fixed. As above, we begin with assuming that $f \in \mathscr{C}(\mathbf{\Lambda}_{-,\infty}\cap \mathbf{\Lambda}_+)$ and define
$$\widehat{H}_{y,v}\::\:t \in [s,\infty) \longmapsto \int_{t-s}^\infty \Theta(t,s,v)f(y-rv,v)\d r$$
and
$$\widehat{\Q}_{y,v} \::\:t \in [s,\infty) \longmapsto \int_{t-s}^\infty \Sigma(t,v)\Theta(t,s,v)f(y-rv,v)\d r.$$
Letting again $\mathcal{L}_{\Sigma(\cdot,v)}$ denote the set of Lebesgue points of $\Sigma(\cdot,v)$, one has as here above that, for any $t \in \mathcal{L}_{\Sigma(\cdot,v)} \cap [s,\infty)$, the application $\widehat{H}_{y,v}$ is derivable in $t$ with
$$\dfrac{\d}{\d t}\widehat{H}_{y,v}(t)=-\widehat{\Q}_{y,v}(t)-\Theta(t,s,v)f(y-(t-s)v,v).$$
Let $T > s$ be fixed. Integrating the above identity over $(s,T)$ one gets, since $f$ is nonnegative
$$\int_s^T \widehat{\Q}_{y,v}(t)\d t + \widehat{H}_{y,v}(T) \leq H(s)=\int_0^\infty f(y-rv,v)\d r.$$
Integrating now over $\mathbf{\Gamma}_{+,\infty}$ we get
\begin{multline*}
\int_s^T \d t \int_{\mathbf{\Gamma}_{+,\infty}}\widehat{\Q}_{y,v}(t)\d\nu_+(y,v) + \int_{\mathbf{\Gamma}_{+,\infty}}\widehat{H}_{y,v}(T)\d\nu_+(y,v)\\
\leq \int_{\mathbf{\Gamma}_{+,\infty}}\d\nu_{+}(y,v)\int_0^\infty f(y-rv,v)\d r.\end{multline*}
Recalling the expression of $\U(t,s)f$ and \eqref{10.49}, this reads
\begin{equation}\label{eq:cas2}\int_s^T \d t \int_{\mathbf{\Gamma}_{+,\infty}}\widehat{\Q}_{y,v}(t)\d\nu_+(y,v) +\int_{\mathbf{\Lambda}_{-,\infty} \cap \mathbf{\Lambda}_+}[\U(T,s)f](x,v)\d x\d\mu(v)
\leq \|f\|_{L^1(\mathbf{\Lambda}_{-,\infty} \cap \mathbf{\Lambda}_+)}.\end{equation}
Now, using the same argument as in the first case, one deduces that, for any $f \in L^1(\mathbf{\Lambda}_{-,\infty} \cap \mathbf{\Lambda}_+)$ and almost every $t \in [s,\infty)$ one has
$$\int_{\mathbf{\Lambda}_{-,\infty} \cap \mathbf{\Lambda}_+} \Sigma(t,v)[\U(t,s)f](x,v)\d x\d\mu(v) < \infty$$
which proves the result for $f \in L^1(\mathbf{\Lambda}_{-,\infty} \cap \mathbf{\Lambda}_+).$\\

\noindent \textit{Third case:}
Assume now that $f \in L^1(\mathbf{\Lambda}_{-,\infty} \cap \mathbf{\Lambda}_{+,\infty}).$ Notice that, since $\mathbf{\Lambda}_{+,\infty}$ is a Borel subset of $\mathbf{\Lambda}$, the associated  indicator function $\mathbbm{1}_{\mathbf{\Lambda}_{+,\infty} \cap \mathbf{\Lambda}_{-,\infty}}$ is a Borel mapping on the product space $\mathbf{\Lambda}=\mathbf{\Omega} \times V$. In particular, the partial applications (in all directions) are measurable, i.e. for any $v \in V$, the application
$$\chi_v\::\:x \in \mathbf{\Omega} \longmapsto \mathbbm{1}_{\mathbf{\Lambda}_{+,\infty} \cap \mathbf{\Lambda}_{-,\infty}}(x,v)$$
is a Borel mapping. In particular, for any $v \in V$,
$$\mathbf{\Omega}_v:=\left\{x \in \mathbf{\Omega}\,,\,\chi_v(x)=1\right\} $$
is a Borel subset of $\mathbf{\Omega}.$ Now, if one assumes that $f$ is nonnegative, one has by Fubini's Theorem
\begin{multline*}\int_{\mathbf{\Lambda}_{+,\infty} \cap \mathbf{\Lambda}_{-,\infty}} f(x,v)\d x\d\mu(v)=\int_{\mathbf{\Lambda}}\mathbbm{1}_{\mathbf{\Lambda}_{+,\infty} \cap \mathbf{\Lambda}_{-,\infty}}(x,v)f(x,v)\d x\d\mu(v)\\
=\int_{V}\left(\int_{\mathbf{\Omega}} \mathbbm{1}_{\mathbf{\Lambda}_{+,\infty} \cap \mathbf{\Lambda}_{-,\infty}}(x,v)f(x,v)\d x\right)\d\mu(v).\end{multline*}
Therefore, for any nonnegative $f \in  L^1(\mathbf{\Lambda}_{-,\infty} \cap \mathbf{\Lambda}_{+,\infty})$ one has
$$\int_{\mathbf{\Lambda}_{+,\infty} \cap \mathbf{\Lambda}_{-,\infty}} f(x,v)\d x\d\mu(v)
=\int_V \left(\int_{\mathbf{\Omega}_v} f(x,v)\d x\right)\d\mu(v).$$
Recalling now that
$$[\U(t,s)f](x,v)=\Theta(t,s,v)f(x-(t-s)v,v) \qquad \forall (x,v) \in \mathbf{\Lambda}_{-,\infty} \cap \mathbf{\Lambda}_{+,\infty}$$
one deduces that
$$\int_{\mathbf{\Lambda}_{+,\infty} \cap \mathbf{\Lambda}_{-,\infty}} [\U(t,s)f](x,v)\d x\d\mu(v)
=\int_V \Theta(t,s,v)\left(\int_{\mathbf{\Omega}_v} f(x-(t-s)v,v)\d x\right)\d\mu(v).$$
For any $v \in V$, the mapping $x \in \mathbf{\Omega}_v  \mapsto y=x-(t-s)v \in \mathbf{\Omega}_v $ is one-to-one and measure-preserving  so that
$$\int_{\mathbf{\Lambda}_{+,\infty} \cap \mathbf{\Lambda}_{-,\infty}} [\U(t,s)f](x,v)\d x\d\mu(v)
=\int_V  \Theta(t,s,v)\left(\int_{\mathbf{\Omega}_v}f(y,v)\d y\right)\d\mu(v).$$
Now, for fixed $v \in V$, according to \eqref{deriveeTheta} one has
$$\int_s^t \Sigma(\t,v) \Theta(\t,s,v)\d\t=-\int_s^t \frac{\partial }{\partial \t}\Theta(\t,s,v)\d \t=1-\Theta(t,s,v) \qquad \text{ for a. e. } t \geq s.$$
Thus,
\begin{multline*}
\int_V \Theta(t,s,v)\left(\int_{\mathbf{\Omega}_v} f(y,v)\d y\right)\d\mu(v)=\int_V\left(\int_{\mathbf{\Omega}_v} f(y,v)\d y\right)\d\mu(v)\\
-\int_V \left(\int_s^t \Sigma(\t,v) \Theta(\t,s,v)\d\t\int_{\mathbf{\Omega}_v}f(y,v)\d y\right)\d\mu(v).\end{multline*}
In other words,
\begin{multline}\label{integUt}
\|\U(t,s)f\|_{L^1(\mathbf{\Lambda}_{+,\infty} \cap \mathbf{\Lambda}_{-,\infty})}=\|f\|_{L^1(\mathbf{\Lambda}_{+,\infty} \cap \mathbf{\Lambda}_{-,\infty})}\\
-\int_s^t\left(\int_V \Sigma(\t,v) \Theta(\t,s,v)\left(\int_{\mathbf{\Omega}_v}f(y,v)\d y\right)\d\mu(v)\right)\d \t\end{multline}
One recognizes as above that, for any $\t \in (s,t)$
\begin{multline*}
\int_V \Sigma(\t,v) \Theta(\t,s,v)\left(\int_{\mathbf{\Omega}_v}f(y,v)\d y\right)\d\mu(v)\\
=\int_V \Sigma(\t,v)\Theta(\t,s,v)\left(\int_{\mathbf{\Omega}_v} f(x-(\t-s)v,v)\d x\right)\d\mu(v)\\
=\int_{\mathbf{\Lambda}_{+,\infty} \cap \mathbf{\Lambda}_{-,\infty}}\Sigma(\t,v)[\U(\t,s)f](x,v)\d x\d\mu(v)
\end{multline*}
so that \eqref{integUt} becomes
\begin{multline}\label{eq-cas3}
\int_s^t \left(\int_{\mathbf{\Lambda}_{+,\infty} \cap \mathbf{\Lambda}_{-,\infty}}\Sigma(\t,v)[\U(\t,s)f](x,v)\d x\d\mu(v)\right)\d\t\\
=\|f\|_{L^1(\mathbf{\Lambda}_{+,\infty} \cap \mathbf{\Lambda}_{-,\infty})}-\|\U(t,s)f\|_{L^1(\mathbf{\Lambda}_{+,\infty} \cap \mathbf{\Lambda}_{-,\infty})}.\end{multline}
In particular,
$$\int_{\mathbf{\Lambda}_{-,\infty} \cap \mathbf{\Lambda}_{+,\infty}} \Sigma(\t,v)[\U(\t,s)f](x,v)\d x \d\mu(v) < \infty \text{ for almost every } \t \in [s,\infty).$$
This achieves the proof.\end{proof}

The proof of Corollary \ref{cor-nonhomo} uses the above computations:

\begin{proof}[Proof of Corollary \ref{cor-nonhomo}] The fact that $\U(t,s)$ maps $\e_+$ into $\D(\mathbf{B}(t))$ for almost every $t\geq s$ and that the mapping $t \in [s,\infty)\mapsto \mathbf{B}(t)\U(t,s)u$ is measurable is a consequence of the previous Lemma exactly as for the proof of  Proposition \ref{propo:homo}. Let us now prove \eqref{dissiKn}. One uses once again the representation provided by Remark \ref{nbSplit} which allows us to distinguish between the three cases: $f\in L^1(\mathbf{\Lambda}_-)$, $f \in L^1(\mathbf{\Lambda_{-,\infty}} \cap \mathbf{\Lambda_+})$ or $f \in L^1(\mathbf{\Lambda_{+,\infty}}\cap \mathbf{\Lambda_{-,\infty}}).$ One uses the notations of Lemma \ref{lem-technint}. First, if $f \in L^1(\mathbf{\Lambda}_-)$, for any $(y,v) \in \Gamma_-$ integrating \eqref{eq:cas1} between with respect to $t$ over $(s,r)$ (for some fixed $r >s$), one gets
$$\int_s^r \Q_{y,v}(t)\d t \leq H_{y,v}(s) - H_{y,v}(r)$$
or, equivalently,
\begin{multline*}
 \int_s^r \left(\int_0^{0 \vee (\tau(y,-v) -(t-s))} \Sigma(t,v)\Theta(t,s,v)f(y+r_0v,v)\d r_0\right)\d t\\
\leq \int_0^{\tau(y,-v)} f(y+r_0v,v)\d r_0 - \int_0^{0 \vee (\tau(y,-v) -(r-s))}  \Theta(r,s,v)f(y+r_0 v,v)\d r_0.
\end{multline*}
Integrating now over $\Gamma_-$ and using \eqref{10.47} we get
$$\int_s^r \|\mathbf{\Sigma}(t)\U(t,s)f\|_{L^1(\mathbf{\Lambda}_-)} \leq \|f\|_{L^1(\mathbf{\Lambda_-})} - \|\U(r,s)f\|_{L^1(\mathbf{\Lambda_-})}.$$
Second, consider $f \in L^1(\mathbf{\Lambda_{-,\infty}} \cap \mathbf{\Lambda_+})$. According to \eqref{eq:cas2} (and recalling the expression of $\widehat{\Q}_{y,v}(t)$ we directly get
$$\int_s^r \|\mathbf{\Sigma}(t)\U(t,s)f\|_{L^1(\mathbf{\Lambda_{-,\infty}} \cap \mathbf{\Lambda_+})} \leq \|f\|_{L^1(\mathbf{\Lambda_{-,\infty}} \cap \mathbf{\Lambda_+})} - \|\U(r,s)f\|_{L^1(\mathbf{\Lambda_{-,\infty}} \cap \mathbf{\Lambda_+})} \quad \forall r >s.$$
Finally, for $f \in L^1(\mathbf{\Lambda_{+,\infty}}\cap \mathbf{\Lambda_{-,\infty}})$, \eqref{eq-cas3} exactly means that
$$\int_s^r \|\mathbf{\Sigma}(t)\U(t,s)f\|_{L^1(\mathbf{\Lambda_{-,\infty}} \cap \mathbf{\Lambda_{+,\infty}})} = \|f\|_{L^1(\mathbf{\Lambda_{-,\infty}} \cap \mathbf{\Lambda_{+,\infty}})} - \|\U(r,s)f\|_{L^1(\mathbf{\Lambda_{-,\infty}} \cap \mathbf{\Lambda_{+,\infty}})} \quad \forall r >s.$$
Summing up all these inequalities we obtain the desired result.
\end{proof}

\end{document}